\documentclass[12pt]{amsart}
\usepackage{amsmath,amssymb, amscd, url}
\usepackage[all]{xy}

\newtheorem{thm}{Theorem}[section]
\newtheorem{lem}[thm]{Lemma}
\newtheorem{cor}[thm]{Corollary}

\newtheorem{prop}[thm]{Proposition}
\newtheorem{question}[thm]{Question}
\newtheorem{ex}[thm]{Example}
\newtheorem*{defn}{Definition}
\newtheorem*{rem}{Remark}

\DeclareMathOperator{\gl}{GL} \DeclareMathOperator{\gsp}{GSp}
\DeclareMathOperator{\mult}{mult} \DeclareMathOperator{\ord}{ord}
\DeclareMathOperator{\id}{id}

\def\P{{\mathbb{P}}}
\def\f{{\mathcal{F}}}
\def\O{{\mathcal{O}}}
\def\p{{\mathfrak{p}}}
\def\Aut{{\rm Aut}}
\def\normal{\lhd}

\def\Z{\mathbb{Z}}
\def\Q{\mathbb{Q}}
\def\R{\mathbb{R}}
\def\F{\mathbb{F}}

\def\G{\mathbb{G}}

\def\Hom{{\rm Hom}}
\def\End{{\rm End}}
\def\Jac{{\rm Jac}}
\def\ord{{\rm ord}}
\newenvironment{Gal}{{\rm Gal\,}}{}

\newenvironment{Frob}{{\rm Frob\,}}{}

\def\im{{\rm im}\,}

\newcommand{\conj}[1]{\overline{#1}}

\def\GL{{\rm GL}}
\def\GSp{{\rm GSp}}
\def\SL{{\rm SL}}
\def\Sp{{\rm Sp}}
\def\SO{{\rm SO}}

\newcommand{\inflim}[1]{\displaystyle \lim_{#1 \rightarrow \infty}}
\def\sp{{\mathop{\rm Sp}}}
\def\sep{{\rm sep}}
\def\units{^\times}
\def\inv{^{-1}}
\def\calf{{\mathcal F}}
\def\calu{{\mathcal U}}
\def\caln{{\mathcal N}}
\def\Z{{\mathbb Z}}
\newcommand{\til}[1]{{\widetilde{#1}}}
\newcommand{\oneover}[1]{\frac{1}{#1}}
\newcommand{\st}[1]{\{#1\}}
\newcommand{\abs}[1]{{\left|#1\right|}}
\newcommand{\rest}[1]{|_{#1}}
\newcommand{\legen}[2]{\genfrac{(}{)}{}{}{#1}{#2}}

\author{Rafe Jones}
\author{Jeremy Rouse}
\address{Department of Mathematics and CS, College of the Holy Cross, Worcester,
 MA 01610}
\address{Department of Mathematics, University of Illinois, Urbana, IL 61801}
\email{rjones@holycross.edu}
\email{jarouse@math.uiuc.edu}

\title[Iterated Endomorphisms]
{Galois Theory of Iterated Endomorphisms}

\thanks{The first author's research was partially supported by NSF grant DMS-0852826.}

\begin{document}

\begin{abstract}
  Given an abelian algebraic group $A$ over a global field $F$,
  $\alpha \in A(F)$, and a prime $\ell$, the set of all preimages of
  $\alpha$ under some iterate of $[\ell]$ generates an extension of $F$ that 
  contains all $\ell$-power torsion points as well as a Kummer-type extension.  We analyze the Galois group of this extension, and for several classes of $A$ we give a simple
  characterization of when the Galois group is as large as possible up to constraints imposed by the endomorphism ring or the Weil pairing.  This Galois group encodes information about the density of primes $\p$ in the ring of integers of $F$ such that the order of $(\alpha \bmod{\p})$ is
  prime to $\ell$.  We compute this density in the general case for
  several classes of $A$, including elliptic curves and one-dimensional tori.  For example, if $F$ is a number field,
  $A/F$ is an elliptic curve with surjective $2$-adic representation
  and $\alpha \in A(F)$ with $\alpha \not\in 2A(F(A[4]))$, then the density of
$\mathfrak{p}$ with ($\alpha \bmod{\p}$) having odd
  order is $11/21$.
\end{abstract}

\maketitle

\section{Introduction} \label{intro} Let $F$ be a global field, $A$ an
abelian algebraic group defined over $F$, $\alpha \in A(F)$, and
$\ell$ a prime.  The tower of extensions $F([\ell^n]^{-1}(\alpha)), n
\geq 1$ contains all $\ell$-power torsion points for $A$, as well as a
Kummer-type extension.  The action of the absolute Galois group
$\Gal(F^{\rm sep}/F)$ on this tower encodes density information about
the orders of reductions $\alpha \bmod \p$, as $\p$ varies over primes
of the ring of integers of $F$ (or $F[C]$ if $F := F(C)$ is the function
field of the affine curve $C$).  (See Theorem \ref{interp}.)  In this
paper we give criteria that ensure the Galois action on the tower
$F([\ell^n]^{-1}(\alpha)), n \geq 1$ is as large as possible given
natural constraints arising from the Weil pairing or endomorphisms not
in $\Z$, and compute the associated density.  We prove results such
as:
\begin{thm} \label{sample} Let $F$ be a number field with ring of integers $\O_F$, $E$ an
  elliptic curve defined over $F$, $\alpha \in E(F)$, and $\ell$ a prime.    Suppose that $\alpha \not\in \ell A(F)$ and the $\ell$-adic Galois representation associated to $E$ surjects onto $\GL_2(\Z_{\ell})$.  If $\ell = 2$, suppose in addition that $\alpha \not\in 2A(F(A[4]))$.  Then the density of primes $\p \subset \O_F$ with $\alpha \bmod{\p}$ having order prime to $\ell$ is
  $$\frac{\ell^{5} - \ell^{4} - \ell^{3} + \ell + 1}{\ell^{5} - \ell^{3}
- \ell^{2} + 1}.$$
\end{thm}

The hypotheses of Theorem \ref{sample} are easy to verify for specific $E$ (see Proposition \ref{surj2prop}).  The $\ell = 2$ case yields the following corollary.
\begin{cor} \label{sampcor} The Somos-4 sequence is defined by 
$a_{0} = a_{1} = a_{2} = a_{3} = 1$ and for $n \geq 4$ by
\[
  a_{n} = \frac{a_{n-1} a_{n-3} + a_{n-2}^{2}}{a_{n-4}}.
\]
The density of primes $p \in \Z$ dividing at least one term of this sequence is 11/21.
\end{cor}
\begin{proof}
Let $E$ be defined by $y^2 + y = x^3 - x$, and let $\alpha = (0,0)$.  Assume for a moment that
\begin{equation} \label{somos}
  [2n-3] \alpha = \left( \frac{a_{n}^{2} - a_{n-1} a_{n+1}}{a_{n}^{2}},
  \frac{a_{n-1}^{2} a_{n+2} - 2 a_{n-1} a_{n} a_{n+1}}{a_{n}^{3}}\right).
  \end{equation}
It follows that $p \mid a_n$ precisely when $[2n-3] \alpha \equiv O \bmod{p}$, which occurs if and only if $\alpha$ has odd order modulo $p$.  In Example~\ref{noncmex} we check the hypotheses of Theorem~\ref{sample} for $E$ and $\alpha$, showing that the density of $p$ such that $\alpha$ has odd order modulo $p$ is 11/21.

To prove \eqref{somos}, one can use the group law on $E$ to reduce \eqref{somos} to 
 the identity
\[
  F(a_{n-1}, a_{n}, a_{n+1}, a_{n+2}) = 0, \text{ where }
F(a,b,c,d) = a^{2} d^{2} - 4abcd + ac^{3} + b^{3} d + b^{2} c^{2}.
\]
It is easy to see that $F(a_{n-1}, a_{n}, a_{n+1}, a_{n+2}) =
\frac{a_{n+2}}{a_{n-2}} F(a_{n-2}, a_{n-1}, a_{n}, a_{n+1})$. Equation \eqref{somos}
now follows by induction and the fact that $F(1,1,1,1) = 0$. 
\end{proof}

We also examine the Galois action on the tower $F([\ell^n]^{-1}(\alpha)), n \geq 1$ in the context of abelian algebraic groups other than elliptic curves.  Our analysis has two components: first, to give explicit conditions on $A$ and $\alpha$ that guarantee the Galois action is as large as possible (subject to natural constraints such as commutativity with the Weil pairing or the action of endomorphism rings larger than $\Z$) and second, to compute the associated density in the case where the Galois action is as large as possible.  

In pursuit of the first goal, we begin by setting $K_\infty$ to be the
union over $n \geq 1$ of the extensions $F([\ell^n]^{-1}(\alpha))$.
The group $\Gal(K_\infty/F)$ acts naturally on the tree of preimages
of $\alpha$ under repeated applications of $\ell$, and thus we refer
to the map
$$\omega : \Gal(F^{\rm sep}/F) \to \Gal (K_\infty/F)$$ as the \textit{arboreal Galois representation} associated to $A$ and $\alpha$.  The image of $\omega$ has as a quotient the usual $\ell$-adic Galois representation attached to $A$, given by the action of Galois on the Tate module $T_\ell(\alpha)$ of $A$.  
The kernel of this quotient map is isomorphic to subgroup of
$T_\ell(\alpha)$ (see p.~\pageref{kummer} for details), and we refer
to it as the \textit{Kummer part} of the image of $\omega$.  The image
of the $\ell$-adic representation has been the subject of much study,
and explicit conditions ensuring surjectivity (up to constraints
imposed by the endomorphism ring or the Weil pairing) are generally
known; here we collect them and give a few additions in the cases
where $\ell$ is 2 or 3.  Generally speaking, surjectivity modulo a low
power of $\ell$ (usually 1) ensures $\ell$-adic surjectivity.  See
Propositions \ref{torirho} (for $A$ a one-dimensional torus),
\ref{surj2prop} (for $A$ and elliptic curve without complex
multiplication), \ref{cmmainprop} (for $A$ an elliptic curve with
complex multiplication), and \ref{abvarrho} (for $A$ a
higher-dimensional abelian variety).  The
study of the surjectivity of the Kummer part originated in
\cite{bachmakov}, \cite{ribet}, and has continued recently thanks in part to
applications to the support problem (see \cite{kowalski} for an
overview).  Our contribution is to give a simple characterization of
when the Kummer part is the full Tate module for certain classes of
$A$ (see Theorems \ref{torikappa}, \ref{surj2}, \ref{cmmain},
\ref{abvarkappa}).  In the latter three theorems, for all $\ell \neq
2$ the condition is just $\alpha \not\in \ell A(F)$.  In the cases we
consider, this makes explicit \cite[Theorem 2, p.~40]{bertrand}, which
states that if $A$ is an abelian variety or the product of an abelian
variety by a torus, then the Kummer part is the full Tate module for
all but finitely many $\ell$ and has open image for all $\ell$ (see
also \cite[Theorem 2.8]{pink} and \cite[Proposition 2.10]{gajda} for
the latter statement).  This result stems essentially from work of
Ribet \cite{ribet}, \cite{ribet-jacquinot}, where it is shown that for $A$
belonging to a large class of commutative algebraic groups, the
modulo-$\ell$ Kummer part is all of $A[\ell]$ for all but finitely
many $\ell$.

Our second goal is to compute, in the case where $\omega$ is
surjective up to natural constraints, the density of $\p$ such that
$\alpha \bmod{\p}$ has order prime to $\ell$.  In Theorem
\ref{matrixprop}, we give a method for computing this density, and
carry out this computation when $A$ is a one-dimensional torus
(Proposition \ref{gmdensity}) or an elliptic curve (Theorem
\ref{gl2den} in the non-CM case, and Theorem \ref{cmcomp} in the CM
case).  For instance, when $A = E$ is an elliptic curve with complex
multiplication, in general $\alpha \bmod{\p}$ has odd order for a set
of $\p$ of density $2/9$ when $2$ splits in the CM ring of $E$, and
$8/15$ when $2$ is inert in the CM ring of $E$.  That the
image of $\omega(\Frob_{\mathfrak{p}})$ encodes $\ell$-power
divisibility properties of $|\alpha \bmod{\p}|$ has already been
established for abelian varieties in \cite{pink} (see also
\cite[Proposition 2.11]{gajda}), where it is shown that various
phenomena occur for all primes in a set of positive Dirichlet
density.
However, 
no densities are computed for specific varieties.  On the other hand,
work originating with Hasse \cite{hasse}, \cite{hasse2} and including Moree
\cite{moree2} and others has led to the computation of all such
densities in the case where $A$ is a trivial one-dimensional torus.


In Section \ref{gen} we develop some general aspects of arboreal
Galois representations for any quasi-projective variety $V$.  In
Section \ref{group} we specialize to the case where $V = A$ is an
abelian algebraic group.  We discuss in detail the Kummer part of the
image of $\omega$, and its relation to the image of the usual
$\ell$-adic representation.  We also give general criteria for the
Kummer part to be the full Tate module (Theorem \ref{surj1}), and show
that when this occurs we can determine $\f(G_{\phi}(\alpha))$ via a
certain matrix computation (Theorem \ref{matrixprop}).  In Section
\ref{torisec} we discuss the case of algebraic tori.  In the case $A =
\G_m$, we reprove certain results of Hasse, Moree and others
\cite{moree2}, and we treat the case where $A$ is a twisted $\G_m$. We
also discuss some examples of higher-dimensional tori.  In Section
\ref{elliptic} we deal with both non-CM and CM elliptic curves.  In
Section \ref{abvar} we treat the case of higher-dimensional abelian
varieties.  Although we are able to give explicit criteria ensuring
that the Kummer part is the full Tate module (Theorem
\ref{abvarkappa}), the complexity of $\GSp_{2d}(\Z_{\ell})$ makes the
computation described in Theorem~\ref{matrixprop} quite difficult to
carry out.  For small $\ell$ we approximate $\f(G_{\phi}(\alpha))$
using MAGMA, and show for instance that if $\dim A = 2$, $\ell = 2$,
and $A, \alpha$ satisfy mild hypotheses, then $0.579 \leq
\f(G_{\phi}(\alpha)) \leq 0.586$.  Thus the density of the set of $\p$
such that $|\alpha \bmod{\p}|$ is odd moves farther from the naive
value of 1/2 in the dimension 2 case.
\begin{question} \label{higherdim}
If we fix say $\ell = 2$ does the limit of
$\f(G_{\phi}(\alpha))$ as the dimension of $A$ grows exist? If so,
what is it?
\end{question}
The first part of Question \ref{higherdim} is answered in the affirmative by Jeff Achter
in the first appendix to this article.  One may also ask whether, if $A$ and $\alpha$ are fixed and
$\ell$ grows, the limit of $\f(G_{\phi}(\alpha))$ must always approach 1.
Finally, we have included a brief appendix of data relating to each example
in the paper.

\section{Preliminaries} \label{gen}

In this section we develop the theory of general arboreal Galois
representations.  While this degree of generality will not be fully
used in the sequel, it provides a framework for the computational
component of the paper.

Let $V$ be a quasiprojective variety and $\phi : V \to V$ be a
finite morphism, both defined over $F$.
Define $U_n$ to be the set of $n$th preimages of
$\alpha$ under the morphism $\phi : V \to V$.  Note that
$T_{\phi}(\alpha) : = \bigsqcup_n U_n$ becomes a rooted tree with
root $\alpha$ when we assign edges according to the action of $\phi$, i.e. $\beta_1$ and $\beta_2$ are
adjacent if and only if $\phi(\beta_1) = \beta_2$.
Moreover, if
$T_{\phi}(\alpha)$ is disjoint from the branch locus
\[
B_{\phi} =
\{\gamma \in V : \#\phi^{-1}(\gamma) < \deg \phi\},
\]
then $U_n$
has $(\deg \phi)^n$ elements and $T_{\phi}(\alpha)$ is the
complete $(\deg \phi)$-ary rooted tree.  This disjointness may be
verified by checking that $\alpha$ is not in $\bigcup_n
\phi^n(B_{\phi})$.

Let $K_n$ be the extension of $F$ obtained by adjoining the
coordinates of the elements of $U_n$, and let $K_{\infty} :=
\bigcup_n K_n$.  Put $\mathcal{G}_n = G_{n,\phi}(\alpha) := \Gal(K_n/F)$, and
note that $\mathcal{G}_n$ is the quotient of $G_{\phi}(\alpha)$ obtained by
restricting the action of $G_{\phi}(\alpha)$ on $T_{\phi}(\alpha)$
to the first $n$ levels of $T_{\phi}(\alpha)$.

We now give a formal definition of $\f(G_{\phi}(\alpha))$.
Note that $G_{\phi}(\alpha)$ is a profinite group and
thus has a natural Haar measure $\mu$, which we take normalized
to have total mass 1.  Define the {\em ends} of
$T_{\phi}(\alpha)$ to be the profinite set $\varprojlim
\{\phi^{-n}(\alpha)\}$ under the natural maps
$\{\phi^{-n}(\alpha)\} \rightarrow \{\phi^{-m}(\alpha)\}$ for $n
> m$ given by $\phi^{n-m}$.
\begin{defn} \label{fdef} 
Assuming the notation above, we set
$$\f(G_{\phi}(\alpha)) := \mu(\{g \in G_{\phi}(\alpha) :
\text{$g$ fixes at least one end of $T_{\phi}(\alpha)$}\}).$$
\end{defn}
\begin{rem}A straightforward argument using the definitions yields
$$\f(G_{\phi}(\alpha)) = \inflim{n} 1/\# \mathcal{G}_n \cdot \#\{g \in \mathcal{G}_n :
\text{$g$ fixes at least one point in $U_n$}\}.$$
This limit exists since the sequence is bounded and monotonically decreasing.  
\end{rem}

A primary consideration in this paper is reduction modulo $\p$ of a
quasiprojective variety and its self-morphisms.  We sketch here what
we mean by this; our discussion is an abbreviated form of that in
\cite[pp.~107-108]{kowalski}.  There exists a reduced scheme
$\mathcal{V} /\O$ of finite type such that $V$ is the generic fiber of
$\mathcal{V}$, as one can see by, loosely speaking, eliminating
denominators in the defining equations of $V$.  We denote by $V_\p$
the fiber of $\mathcal{V}$ over $\p$, and by $f_{\p}$
the finite field $\O/\mathfrak{p}$. Given $\alpha \in V(F)$ and a
finite morphism $\phi : V \to V$, for all but finitely many $\p$ the
following hold: $V_\p$ is quasiprojective, there is a reduction
$\alpha_\p \in V_\p(f_{\p})$ of $\alpha$ that is independent of the choice of
$\mathcal{V}$, and there is a reduced morphism $\overline{\phi} :
V_{\p} \rightarrow V_{\p}$ with $\deg \overline{\phi} = \deg \phi$.

In this section we show that $\f(G_{\phi}(\alpha))$
encodes certain dynamical information about $\overline{\alpha}$
under $\overline{\phi}$ as $\p$ varies over the finite primes of
$F$.  By the density of a set $S$ of primes of $F$, we mean the
Dirichlet density
\begin{equation} \label{densedef}
D(S) = \lim_{s \rightarrow 1^{+}} \frac{\sum_{\p \in S} \; N(\p)^{-s}}
{\sum_{\p} \; N(\p)^{-s}},
\end{equation}
where $N(\p)$ denotes the norm of $\p$.  Note that the limit above does not 
exist for all sets of primes, and so we define the upper density $D^+(S)$ to be the expression in \eqref{densedef} with 
$\lim$ replaced by $\limsup$.   There is a stronger notion called natural density, given by 
$$d(S) = \inflim{n} \#\{\p \in S : N(\p)
\leq n\}/\#\{\p : N(\p) \leq n\}.$$  In the case where $F$ is a
number field, the results of this paper hold with $D(S)$ replaced by $d(S)$, due
to the Chebotarev density theorem (in the function field case, Chebotarev's theorem requires additional hypotheses to give results about natural density).  

Before stating the main result of this section, we give some
terminology.  If $S$ is a set, $f: S \rightarrow S$ is a map, and
$f^n$ the $n$th iterate of $f$, we say that $s \in S$ is {\em
  periodic} under $f$ if $f^n(s) = s$ for some $n \geq 1$.  We say
that $s$ is {\em preperiodic} if $s$ is not periodic but $f^n(s) =
f^m(s)$ for some $n,m \geq 1$.  Note that if $S$ is finite then
every point in $S$ is either periodic or preperiodic.

\begin{prop} \label{maingen} Assume the notation above, and let 
$$S = \{\p \subset \mathcal{O} : \text{$\overline{\alpha} \in
    V(f_{\p})$ is periodic under $\overline{\phi}$}\}.$$
    Then $\mathcal{F}(G_{\phi}(\alpha)) \geq D^+(S)$.   In particular, if $D(S)$ exists then $\mathcal{F}(G_{\phi}(\alpha)) \geq D(S).$
\end{prop}
\begin{rem}
In Theorem \ref{interp} we give conditions that imply $D(S)$ exists and that the inequality is an equality.
\end{rem}

\begin{proof}
Let ${\rm Per}(\phi, \alpha) = \{\p \subset \O :
\text{$\overline{\phi}^n(\overline{\alpha}) = \overline{\alpha}$ for
  some $n \geq 1$}\}$.  We begin by showing that $\p \in {\rm
  Per}(\phi, \alpha)$ if and only if for each $n$ there is $\gamma \in
V(f_{\p})$ such that $\overline{\phi}^n(\gamma) = \overline{\alpha}$.
If $\overline{\phi}^m(\overline{\alpha}) = \overline{\alpha}$ for some
$m$, then for any $n$ we may write $n = mk + r$ with $0 \leq r < m$
and take $\gamma = \overline{\phi}^{m-r}(\overline{\alpha})$. To
show the reverse inclusion, the finiteness of $V(f_{\p})$ implies that
there exist $n_2 > n_1$ and $\gamma$ such that
$\overline{\phi}^{n_1}(\gamma) = \overline{\phi}^{n_2}(\gamma) =
\overline{\alpha}$.  Then $$\overline{\phi}^{n_2 -
  n_1}(\overline{\alpha}) = \overline{\phi}^{n_2 -
  n_1}(\overline{\phi}^{n_1}(\gamma)) = \overline{\phi}^{n_2}(\gamma)
= \overline{\alpha}.$$

Now let
$$\Omega_n =
\{\p : \p \text{ is unramified in $K_n$ and $\overline{\phi}^n(x) =
\overline{\alpha}$
has no solution in $V(f_{\p})$ } \}.$$
If $\p \in \Omega_n$, then by the previous paragraph, clearly $\p
\not\in {\rm Per}(\phi, \alpha)$.  Since only finitely many primes
ramify in $K_n$, we have
\begin{equation} \label{bound}
D^+({\rm Per}(\phi, \alpha)) \leq 1- D^+(\Omega_n).
\end{equation}

Let $G_F$ be Galois group of the separable closure of $F$, and let
$\Frob_{\p} \subset G_F$ be the Frobenius conjugacy class at ${\p}$.
By the Chebotarev density theorem, the density of $\p$ with
$\Frob_{\p}$ having prescribed image $C \subseteq \mathcal{G}_n$ exists and is
$\#C/\#\mathcal{G}_n$.

Let $\p$ be a prime of $F$ not ramifying in $K_n$ and such that
$\deg \overline{\phi} = \deg \phi$; this excludes only a finite
number.  There exists $\gamma \in V(f_{\p})$ such that
$\overline{\phi}^n(\gamma) = \overline{\alpha}$ if and only if
the action of $\Frob_{\p}$ on $U_n$ has a fixed point.  By the
Chebotarev density theorem the density of such $\p$ exists and equals
$$\#\{\sigma \in \mathcal{G}_n : \text{$\sigma$ fixes at
least one element of
  $U_n$}\}/\#\mathcal{G}_{n}.$$ Let us denote this quantity by $d_n$, and note 
that $d_n = D(\Omega_n^c) = 1 - D(\Omega_n)$ (so in particular $D(\Omega_n)$ exists).  By \eqref{bound} we now have $D^+({\rm Per}(\phi,
\alpha)) \leq \inflim{n} d_n,$ and this last limit is just
$\mathcal{F}(G)$.
\end{proof}

We close this section with some general remarks about arboreal
representations.  A natural
question to ask is when $G_{\phi}(\alpha)$ must have finite index in $\Aut(T_{\phi}(\alpha))$, where the latter indicates the full group of tree automorphisms.  Certainly this need not happen all the time, as the examples in the rest of this paper show: when $V$ has a group structure, automorphisms of $T_{\phi}(\alpha)$ failing to commute with the group law cannot be Galois elements, preventing $G_\phi(\alpha)$ from being a large subgroup of $\Aut(T_{\phi}(\alpha))$.  There are, however, situations where 
$G_{\phi}(\alpha) \cong \Aut(T_{\phi}(\alpha))$, such as when $F = \Q$, $V = \mathbb{P}^1$, 
$\phi = x^2 + 1$, and $\alpha = 0$ \cite{stoll}.  Indeed, a similar result holds for infinitely many $\phi$ in the family $x^2 + c$ \cite{stoll}, though even in this family open questions remain: for $c = 3$, 
$|\Aut(T_{\phi}(\alpha)) : G_{\phi}(\alpha)| \geq 2$, and the index is not known to be finite.  Less is known about the more general question of whether $|\Aut(T_{\phi}(\alpha)) : G_{\phi}(\alpha)|$ must be finite in the case $F = \Q$ and $V = \mathbb{P}^1$.  The  
first author has shown finite index for $\phi$ belonging to two infinite families of quadratic polynomials \cite[Section 3]{quaddiv}, but otherwise the question remains open.  For a related discussion, see \cite{serreimage}.  We note that in the case where $|\Aut(T_{\phi}(\alpha)) : G_{\phi}(\alpha)|$ is finite, we have $\mathcal{F}(G) = 0$; this follows from natural generalizations of \cite[Section 5]{galmart}.

\section{Arboreal representations associated to abelian algebraic groups}
\label{group}

In this section, we specialize to the case where $V = A$ is an abelian algebraic
group and $\phi$ is multiplication by a prime $\ell$. 
We first give an interpretation of
$\f(G_{\phi}(\alpha))$ in this case, then we describe the Galois
groups $\mathcal{G}_n := \Gal(F(U_n)/F)$ in terms of the groups $A[\ell^{n}]
:= \{ x \in A : \ell^{n} x = 0 \}$ 
and their automorphism groups. We show that the
image $G_{\phi}(\alpha)$ of $\omega = \omega_{\phi,\alpha} : \Gal(F^{\sep}/F) \to
\Aut(T_{\phi}(\alpha))$ lands inside a particular semi-direct
product, and fits into a short exact sequence with the Kummer part and the image
of the $\ell$-adic representation.  Moreover, we
give criteria for the Kummer part to be the full Tate module.

We fix $\alpha \in A(F)$, and we refer to $G_{\phi}(\alpha)$,
$T_{\phi}(\alpha)$ and $\mathcal{F}(G_{\phi}(\alpha))$ as $G$, $T$,
and $\mathcal{F}(G)$, respectively. We denote the group operation on
$A$ additively. We assume that $\phi = [\ell]$ has degree $\ell^{d}$
and is finite and separable.  This implies that $\phi$ has no branch
points, and that the extensions $K_{n}/F$ are Galois.  It also implies
that $A[\ell^{n}] \cong (\Z/\ell^{n}
\Z)^{d}$ for all $n \geq 1$.

The map $\omega$ has a natural decomposition into two parts because \label{kummer}
$A$ is an abelian algebraic group, and we now give some notation and
terminology that we use throughout the sequel.  Let $T_{\ell}(A)
:= \varprojlim A[\ell^{n}]$ be the Tate module of $A$.  Note that in
the notation of Section \ref{gen}, $T_{\ell}(A)$ is the same as
$T_{\phi}(O)$, where $O \in A$ is the identity.  We make one change in notation: since 
$T_\ell(A)$ has a group structure, we use $\Aut(T_{\ell}(A))$ to denote the set of group automorphisms; before we used $\Aut(T_{\phi}(\alpha))$ to denote the set of tree automorphisms.  

\begin{defn} \label{betan}
For each $n \geq 1$, let $\beta_{n} \in U_{n}$ be a chosen element
so that $\phi(\beta_{n}) = \beta_{n-1}$, with $\beta_{0} = \alpha$.
Define $$\omega_{n} : \Gal(K_{n}/F) \to A[\ell^{n}] \rtimes \Aut(A[\ell^n])$$ by
$\omega_{n}(\sigma) := (\sigma(\beta_{n}) - \beta_{n},
\sigma|_{A[\ell^{n}]})$.  Passing to the inverse limit gives
$\omega : \Gal(K_{\infty}/F) \to T_{\ell}(A) \rtimes \Aut(T_{\ell}(A)).$
\end{defn}

For the remainder of the paper, we will use the following notation. 

\medskip

\begin{tabular}{cl}
Notation & Meaning\\
\hline
$F$ & Base field\\
$T_{n}$ & $F(A[\ell^{n}])$\\
$F_{n}$ & $F(\beta_{n})$\\
$K_{n}$ & $T_{n} F_{n}$\\
$T_{\infty}$ & $\bigcup_{n=1}^{\infty} T_{n}$\\
$K_{\infty}$ & $\bigcup_{n=1}^{\infty} K_{n}$\\
$\mathcal{T}_{n}$ & $\Gal(T_{n}/F)$, the torsion part\\
$\mathcal{K}_{n}$ & $\Gal(K_{n}/T_{n})$, the Kummer part\\
$\mathcal{G}_{n}$ & $\Gal(K_{n}/F)$\\
$\rho$ & $\Gal(T_{\infty}/F) \to \Aut(T_{\ell}(A))$, the torsion representation\\
$\mathcal{T}$ & $\varprojlim \mathcal{T}_{n}$, the image of $\rho$\\
$\kappa$ & $\Gal(K_{\infty}/T_{\infty}) \to T_{\ell}(A)$, the Kummer map\\
$\omega$ & $\Gal(K_{\infty}/F) \to T_{\ell}(A) \rtimes \Aut(T_{\ell}(A))$,
the arboreal representation\\
\end{tabular}


\medskip

The next proposition says that these two parts give us full information
about the image of $\kappa$.  It is closely related to \cite[p.~5]{pink}.

\begin{prop}  \label{inj}
Assume the notation above. For $n \geq 1$, $\omega_{n}$ is an injective
homomorphism.
\end{prop}
\begin{proof}
For $\sigma, \tau \in \Gal(K_{n}/F)$, we have
\begin{align*}
  \omega_{n}(\sigma \tau) &= (\sigma \tau(\beta_{n}) - \beta_{n},
  \sigma \tau|_{A[\phi^{n}]})\\
  &= (\sigma(\tau(\beta_{n})) - \sigma(\beta_{n}) + \sigma(\beta_{n})
  - \beta_{n},
  \sigma|_{A[\phi^{n}]} \tau|_{A[\phi^{n}]})\\
  &= ((\sigma(\beta_{n}) - \beta_{n}) + \sigma(\tau(\beta_{n}) - \beta_{n}),
  \sigma|_{A[\phi^{n}]} \tau|_{A[\phi^{n}]})\\
  &= (\sigma(\beta_{n}) - \beta_{n}, \sigma)
  (\tau(\beta_{n}) - \beta_{n}, \tau)\\
  &= \omega_{n}(\sigma) \omega_{n}(\tau).
\end{align*}
Thus, $\omega_{n}$ is a homomorphism. Suppose that $\sigma \in \ker
\omega_{n}$. Then, $\sigma(\beta_{n}) - \beta_{n} = 0$ so
$\sigma(\beta_{n}) = \beta_{n}$. Moreover,
$\sigma|_{A[\ell^{n}]}$ is the identity. Thus, if $\beta \in
U_{n}$ we have $\ell^{n} \beta = \alpha$ and so
\[
  \ell^{n}(\beta - \beta_{n}) = \alpha - \alpha = 0.
\]
Thus, $\beta - \beta_{n} \in A[\ell^{n}]$. Hence, $\sigma(\beta -
\beta_{n}) = \beta - \beta_{n}$. It follows that $\sigma(\beta) =
\beta$. Thus $\sigma$ fixes $U_{n}$ and hence $K_{n}$, proving that
$\sigma = 1$ and $\omega_{n}$ is injective.
\end{proof}

We summarize the preceding discussion and Proposition with the
following commutative diagram:
\[
\xymatrix{
1 \ar[r] & \Gal(K_{\infty}/T_{\infty}) \ar[r]
\ar[d]^\kappa & \Gal(K_{\infty}/F) \ar[r] \ar[d]^\omega &
\Gal(T_{\infty}/F) \ar[r] \ar[d]^\rho & 1\\
1 \ar[r] & T_{\ell}(A) \ar[r] &
T_{\ell}(A) \rtimes \Aut(T_{\ell}(A))
  \ar[r] & \Aut(T_{\ell}(A)) \ar[r] & 1\\
}
\]
The rows are exact, the maps on the top row being the natural
ones. The nontrivial maps on the bottom row are inclusion into
the first factor, and projection onto the second factor,
respectively.  The vertical arrows are all injections.  For each
$n$ one has a corresponding diagram modulo $\ell^n$, with the
vertical maps being $\kappa_n, \omega_n$, and $\rho_n$.  Finally, for the 
remainder of the article we regard $\omega$ as mapping into
$T_{\ell}(A) \rtimes \Aut(T_{\ell}(A))$, rather than into the full automorphism 
group of the tree $T_{\ell}(\alpha).$  

\begin{thm} \label{interp} Let $G = \Gal(K_\infty / F)$, and let
  $\f(G)$ be as defined on p.~\pageref{fdef}.  Let $A$ be the product
  of an abelian variety by a torus, $\phi = [\ell]$, and assume the
  orbit of $\alpha \in A(F)$ under $[\ell]$ is Zariski dense in $A$.
  Then the set
$$\{\mathfrak{p} \subset \O_F : \text{the order of $\overline{\alpha} \in A(f_{\p})$ is not divisible
by $\ell$}\}$$ has a Dirichlet density, and it is given by $\f(G)$.
\end{thm}

\begin{rem}[(As pointed out to the authors by A. Perucca)]
If the orbit of $\alpha$ is not dense in $A$, the result still holds in many circumstances.  In particular, if $A_\alpha$ is the smallest $F$-algebraic subgroup of $A$ containing $\alpha$, and the number of connected components of $A_\alpha$ is prime to $\ell$, the result holds.  One achieves this by replacing $\alpha$ with a multiple of itself to get that $A_\alpha$ is connected, then applying \cite[Proposition 2.5]{perucca} and proceeding with the proof below.  In the case where the number of connected components of $A_\alpha$ is divisible by $\ell$, it follows from \cite[Main Theorem 1]{perucca} that the set in Theorem \ref{interp} contains only finitely many primes. 
\end{rem}

\begin{proof} 
We begin by noting that the hypothesis that the orbit of $\alpha$ is dense in $A$ permits us to apply a theorem of Bertrand (\cite[Theorem 2, p.~40]{bertrand}) showing that $\im \kappa$ has finite index in $T_{\ell}(A)$.  

Let ${\rm Per}(\ell, \alpha) = \{\p \subset \O_F :
\text{$[\ell^n](\overline{\alpha}) = \overline{\alpha}$ for
  some $n \geq 1$}\}$.  Denote the order of
  $\overline{\alpha} \in A(f_{\p})$ by $m$.  We show that $\ell \nmid
  m$ if and only if $\overline{\alpha}$ is periodic under $\ell$.  If
  $\ell \nmid m$ then $\ell \in (\Z/m\Z)^{\times}$, whence $\ell^n
  \equiv 1 \bmod{m}$ for some $n$.  Thus $[\ell^n] \overline{\alpha} =
  \overline{\alpha}$, whence $\overline{\alpha}$ is periodic under
  $\ell$.  Conversely, if $[\ell^n] \overline{\alpha} =
  \overline{\alpha}$ for some $n$, then $[\ell^n - 1]
  \overline{\alpha} = \overline{0}$, whence $\ell$ cannot divide the
  order of $\overline{\alpha}$.
  
  Next, let $\Frob_\p$ denote the Frobenius conjugacy class at $\p$ in $\Gal(F^{\rm sep}/F)$, and let 
  $t_{\p,n}$ denote its image in $\mathcal{T}_n$.  Define  
  \begin{eqnarray*}
  NP_n & := & \{\p \subset \O : \text{$\Frob_\p$ has no fixed points in $U_n$}\} \\
  P_n & := & \{\p \subset \O : \text{$\Frob_\p$ has a fixed point in $U_n$ and $\det(t_{\p,n} - \rm{id}) \neq 0$}\}.
  \end{eqnarray*}
  
By the Chebotarev density theorem, $D(P_n)$ and $D(NP_n)$ exist for all $n \geq 1$.  We note that the extension $K_{\infty}/F$ is ramified over only finitely many primes, whence for all but finitely many $\p$ there is a well-defined action of $\Frob_\p$ on $U_m$ for all $m \geq n$ (this will be used in the next paragraph).  To prove finite ramification, it is enough to show that except for finitely many primes, the elements of $U_n$ remain distinct under reduction modulo $\p$.  This follows from the fact that reduction modulo $\p$ is injective on $A[\ell^n]$ for all $n$, except for finitely many $\p$, a statement that holds for any connected abelian algebraic group (see \cite[Lemma 4.4]{kowalski}).
  
Next, it follows from the proof of Proposition \ref{maingen} that the complement of ${\rm Per}(\ell, \alpha)$ contains $\bigcup_{n \geq 1} NP_n$.   We claim that $\bigcup_{n \geq 1} P_n \subseteq {\rm Per}(\ell, \alpha)$.
Identify $U_n$ with $A[\ell^n]$ via the map $\beta_n + \gamma \mapsto \gamma$, and note that under this identification $\Frob_\p$ acts on $U_n$ as $t_{\p,n} + {\bf v}$, where $t_{\p,n} \in \GL_m(\Z/\ell^n\Z)$ and ${\bf v} = \sigma(\beta_n) - \beta_n \in A[\ell^n]$ (see Proposition \ref{inj}).  If $\p \in P_n$ for some $n$, then this action has a fixed point, so ${\bf v} \in \im(t_{\p,n} - {\rm id})$.  Since $\ord_{\ell} (\det(t_{\p,n} - {\rm id})) < n$, the $\Z/\ell^n\Z$-submodule $\im(t_{\p,n} - {\rm id})$ has index $\ell^{\ord_{\ell} (\det(t_{\p,n}))}$ in $A[\ell^n]$, and the same statement holds if $n$ is replaced by any larger integer. Hence every lift of ${\bf v}$ to $A[\ell^m]$ for $m > n$ must be in $\im(t_{\p,m} - {\rm id})$, and it follows that $\Frob_p$ acts on $U_k$ with a fixed point for all $k \geq 1$.  By the first paragraph of the proof of Proposition \ref{maingen}, this implies $\p \in {\rm Per}(\ell, \alpha)$.  
 
Finally, for fixed $n$, the set of $\p$ not belonging to $P_n$ or $NP_n$ is 
$$E_n := \{\p \in \O : \text{$\Frob_\p$ has a fixed point in $U_n$ and $\det(t_{\p,n} - {\rm id}) = 0 \bmod{\ell^n}$}\}.$$ 
By the Chebotarev Density theorem, $D(E_n)$ is given by 
$$\frac{1}{\#\mathcal{G}_n} \cdot \#\{g \in \mathcal{G}_n : \text{$g$ has a fixed point in $U_n$ and $\det(g |_{T_n} - {\rm id}) = 0 \bmod{\ell^n}$}\}.$$
If $\det(g |_{T_n} - {\rm id}) = 0 \bmod{\ell^n}$, then the index of $\im(g |_{T_n} - {\rm id})$ in $A[\ell^n]$ is at least $\ell^n$.  Hence 
$$\#\{h \in \mathcal{G}_n : \text{$h |_{T_n} = g |_{T_n}$ and $h$ has a fixed point in $U_n$}\} \leq \frac{1}{\ell^n} \cdot \#A[\ell^n].$$
By \cite[Theorem 2, p.~40]{bertrand}, the subgroup of $h \in G$ with $h |_{T_\infty} = g |_{T_\infty}$ has finite index $m$ in $T_\ell(\alpha)$.  For $n$ sufficiently large, this implies that 
$\#\{h \in \mathcal{G}_n : h |_{T_n} = g |_{T_n}\} = (1/m) \cdot \#A[\ell^n]$.  We now have shown the proportion of $h \in \mathcal{G}_n$ with $h |_{T_n} = g |_{T_n}$ and that fix a point in $A[\ell^n]$ is at most $m/\ell^n$.  It follows that $D(E_n) \leq m/\ell^n$.  

 

This gives ${\displaystyle \lim_{n \to \infty} (D(P_n) + D(NP_n)) = 1}$.  Let $L_s$ denote the lim sup of the expression in \eqref{densedef} for $S = {\rm Per}(\ell, \alpha)$, and $L_i$ denote the corresponding lim inf.  Since $\bigcup_{n \geq 1} P_n \subseteq {\rm Per}(\ell, \alpha)$, we have $L_i \geq {\displaystyle \lim_{n \to \infty} D(P_n)}$.  Since the complement of ${\rm Per}(\ell, \alpha)$ contains $\bigcup_{n \geq 1} NP_n$, we have $L_s  \leq 1 - {\displaystyle \lim_{n \to \infty} D(NP_n)}$.  Hence $L_s = L_i = 1 - {\displaystyle\lim_{n \to \infty} D(NP_n)}$.  This last expression is the same as $\f(G)$.  \end{proof}

We now work toward a theorem that will allow us to determine
information about the image of $\omega$. If $G$ is any profinite
group, we let $\Phi(G)$ denote its Frattini subgroup, namely the
intersection of all maximal open subgroups of $G$. Properties of
the Frattini subgroup of $\mathcal{T}$ will be important for
determining the image of $\omega$.  

\begin{thm} \label{frattini}
If $G \leq \GL_{d}(\Z_{\ell})$ is a profinite group, then $[G : \Phi(G)]$
is finite.
\end{thm}
\begin{proof}
Let $N \leq G$ be the kernel of the map $G \to \GL_{d}(\Z/\ell \Z)$. 
Since $\Phi(N) \leq \Phi(G)$ and $N$ has finite index in $G$, it suffices
to show that $[N : \Phi(N)]$ is finite. For $n \geq 1$, let $N^{(n)}$ be the
kernel of the map $N \to \GL_{d}(\Z/\ell^{n} \Z)$ and define
$\delta_{n} : N^{(n)}/N^{(n+1)} \to M_{d}(\Z/\ell \Z)$ by $\delta_{n}(g)
= \frac{g-1}{\ell^{n}}$. It is easy to see that $\delta_{n}$ is an injective
homomorphism. This implies that $N$ is a pro-$\ell$ group, and hence
$\Phi(N) = N' N^{\ell}$. We will show that $[N : N^{\ell}]$ is finite.

If $g = I + \ell^{n} M$, then
\[
  g^{\ell} = \sum_{k=0}^{\ell} \binom{\ell}{k} \ell^{nk} M^{k}
  \equiv I + \ell^{n+1} M \pmod{\ell^{2n}}.
\]
If $\ell > 2$, then the above congruence holds modulo $\ell^{2n+1}$.
It follows that for $n \geq 2$ or $n = 1$ and $\ell > 2$, we have
$g^{\ell} \in N^{(n+1)}$ and $\delta_{n+1}(g^{\ell}) = \delta_{n}(g)$.
It follows that we have the increasing sequence
\[
  \delta_{n}(N^{(n)}/N^{(n+1)}) \subseteq \delta_{n+1}(N^{(n+1)}/N^{(n+2)}) 
\subseteq \cdots
\]
where all the groups are contained in $M_{d}(\Z/\ell \Z)$. Hence,
there is some $m$ so that the $\ell$th power map 
$N^{(n)}/N^{(n+1)} \to N^{(n+1)}/N^{(n+2)}$ is surjective for $n \geq m$.
This implies that if $g \in N^{(m+1)}$, then $g$ can be written as a product
of $\ell$th powers in every quotient $N^{(m+1)}/N^{(m+k)}$. The fact that
the $N^{(n)}$ form a base for the open neighborhoods of the identity then
imply that $\Phi(N) \supseteq N^{\ell} \supseteq N^{(m+1)}$,
and so $\Phi(N)$ has finite index in $N$, as desired.
\end{proof}

Our next goal is to develop criteria that will ensure that $\im \omega_{n}
\cong A[\ell^{n}] \rtimes \mathcal{T}_{n}$. 

\begin{thm}
\label{surj1} Let the notation be as above. Suppose that for some $m \geq 1$
the following hold.
\begin{enumerate}
\item $A[\ell^{m}]/A[\ell^{m-1}]$ is irreducible as a $\mathcal{T}_{m}$-module.
\item $\alpha \not\in A(F) \cap \ell A(T_{n})$ for all $n \geq m$.
\end{enumerate}
Then $\im \omega_{n} \cong A[\ell^{n}] \rtimes \mathcal{T}_{n}$ for all 
$n \geq m$.
\end{thm}
\begin{proof}
Recall $\mathcal{K}_{n} = \Gal(T_{n}(\beta_{n})/T_{n}) =
\im \omega_{n} \cap A[\ell^{n}]$. If $(a,X) \in
A[\ell^{n}] \rtimes \mathcal{T}_{n}$, and $(b,1) \in \im \omega_{n} \cap
A[\ell^{n}]$, one can compute that
\[
  (a,X) (b,1) (a,X)^{-1} = (Xb, 1),
\]
and hence $\mathcal{K}_{n}$ has the structure of a $\mathcal{T}_{n}$-module.
It suffices to show that $\mathcal{K}_{n} = A[\ell^{n}]$ for $n \geq m$,
since in this case, if $(a,X)$ is an arbitrary element of 
$A[\ell^{n}] \rtimes \mathcal{T}_{n}$ then there is some 
$(b,X) \in \im \omega_{n}$ and $(a-b,1) \in \mathcal{K}_{n}$ and we have
\[
  (a-b,1) (b,X) = (a,X) \in \im \omega_{n},
\]
and the desired result holds.

To show that $\mathcal{K}_{n} = A[\ell^{n}]$, we prove two things.
First, if $M \leq A[\ell^{n}]$ is any $\mathcal{T}_{n}$-submodule,
then either $M = A[\ell^{n}]$ or $M \leq A[\ell^{n-1}]$. Finally,
we'll prove that $\mathcal{K}_{n} \leq A[\ell^{n-1}]$ does not occur.

We prove the first of the two statements above by induction on $n$.
For the base case $n = m$, we have $A[\ell^{m}]/A[\ell^{m-1}]$ is
irreducible as a $\mathcal{T}_{m}$-module, and so the homomorphism $M
\to A[\ell^{m}]/A[\ell^{m-1}]$ is either trivial or surjective.  In
the first case, $M \leq A[\ell^{m-1}]$ and in the second, $M$ contains
a complete set of representatives of $A[\ell^{m}]/A[\ell^{m-1}]$. 
The latter fact, together with the observation that $\ell M \leq M$ implies that
$M = A[\ell^{m}]$.

Suppose now that $n > m$ and $M$ is a $\mathcal{T}_{n}$-submodule of
$A[\ell^{n}]$. Then, $\ell M$ is a $\mathcal{T}_{n-1}$-submodule of
$A[\ell^{n-1}]$. Induction now implies that $\ell M \leq
A[\ell^{n-2}]$ or $\ell M = A[\ell^{n-1}]$.  In the first case, $M
\leq A[\ell^{n-1}]$, while in the second the following commutative
diagram with exact rows, together with the five-lemma, implies that $M
= A[\ell^{n}]$.

\[
\begin{CD}
0 @>>> A[\ell] @>>> M @>\ell>> A[\ell^{n-1}] @>>> 0\\
@|       @|         @VVV         @|               @|\\
0 @>>> A[\ell] @>>> A[\ell^{n}] @>\ell>> A[\ell^{n-1}] @>>> 0\\
\end{CD}
\]

Finally, we will prove that $\mathcal{K}_{n} = A[\ell^{n}]$.
We have the following diagram with exact rows:
\[
\begin{CD}
0 @>>> A[\ell^{n}] @>>> A @>\ell>> A @>>> 0\\
@. @VV \ell^{n-1} V @VV \ell^{n-1} V @|\\
0 @>>> A[\ell] @>>> A @>\ell>> A @>>> 0
\end{CD}
\]
This gives rise to the following diagram with exact rows:
\[
\begin{CD}
0 @>>> A(T_{n})/\ell^{n} A(T_{n}) @>\delta_{n}>> H^{1}(T_{n},A[\ell^{n}])\\
@. @VVV @VV \ell^{n-1} V\\
0 @>>> A(T_{n})/\ell A(T_{n}) @>\delta_{1}>> H^{1}(T_{n},A[\ell])
\end{CD}
\]
Here $\delta_{n}(\alpha)$ is the element of $H^{1}(T_{n},A[\ell^{n}])$
represented by the 1-cocycle $\sigma \mapsto \sigma(\beta_{n}) -
\beta_{n}$.  Since $\mathcal{K}_{n} = \Gal(T_{n}(\beta_{n})/T_{n})$,
it follows that if $\mathcal{K}_{n} \leq A[\ell^{n-1}]$, then
$\delta_{n}(\alpha)$ lies in the kernel of $\ell^{n-1} :
H^{1}(T_{n},A[\ell^{n}]) \to H^{1}(T_{n},A[\ell])$.  This implies that
$\delta_{1}(\alpha) = 0$, which by the diagram above implies that
$\alpha \in A(F) \cap \ell A(T_{n})$. This contradicts the second
assumption of the theorem.
 
Thus, $\mathcal{K}_{n} = A[\ell^{n}]$ and 
$\im \omega_{n} \cong A[\ell^{n}] \rtimes \mathcal{T}_{n}$.
\end{proof}

We will see that in most cases, the conditions of Theorem~\ref{surj1}
are satisfied with $m = 1$. The next three lemmas will deal with
establishing condition 2 of Theorem~\ref{surj1} under suitable hypotheses.

\begin{lem}
\label{surjlem1}
Suppose that $A[\ell]$ is irreducible as a $\mathcal{T}_{1}$-module,
$\alpha \not\in \ell A(T_{1})$ and $\alpha \in \ell A(T_{n})$ for some $n \geq 
2$. Then $F(\beta_{1}) \subseteq T_{n}$ and $\Gal(T_{n}/F(\beta_{1}))$
is a maximal subgroup of $\Gal(T_{n}/F)$. 
\end{lem}
\begin{proof}
The assumption that $\alpha \not\in \ell A(T_{1})$ implies that
$\beta_{1} \not\in A(T_{1})$ and so $\Gal(K_{1}/T_{1}) = A[\ell]$ and
$\mathcal{G}_{1} = \Gal(K_{1}/F) \cong A[\ell] \rtimes \mathcal{T}_{1}$. 
If $\mathcal{T}_{1} \leq N \leq \mathcal{G}_{1}$ is any subgroup, then
$N \cap A[\ell]$ is a $\mathcal{T}_{1}$-submodule, and hence $N = \mathcal{T}_{1}$ or $N = \mathcal{G}_{1}$. Thus, $\mathcal{T}_{1}$ (whose fixed field is
$F(\beta_{1}))$ is a maximal subgroup of $\mathcal{G}_{1}$, as desired. This
implies that there are no fields that lie between $F$ and $F(\beta_{1})$
and hence $\Gal(T_{n}/F(\beta_{1}))$ is a maximal subgroup of $\Gal(T_{n}/F)$
for any $n$ with $F(\beta_{1}) \subseteq T_{n}$.
\end{proof}

Let $N^{(n)} = \Gal(T_{\infty}/T_{n})$. Via the embedding
of $\mathcal{T}_{n+1} \to \GL_{d}(\Z/\ell^{n+1} \Z)$, we have 
that $N^{(n)}/N^{(n+1)} \cong \{ M \in \mathcal{T}_{n+1} : M \equiv 1 \pmod{\ell^{n}}
\}$. The group $\mathcal{T}_{1}$ acts by conjugation on $N^{(n)}/N^{(n+1)}$
and hence $N^{(n)}/N^{(n+1)}$ has the structure of a $\mathcal{T}_{1}$-module.

\begin{lem}
\label{surjlem2}
If $n \geq 1$ and 
$\Hom_{\mathcal{T}_{1}}(N^{(n)}/N^{(n+1)}, A[\ell]) = 0$, then 
$A(F) \cap \ell A(T_{n}) = A(F) \cap \ell A(T_{n+1})$.
\end{lem}
\begin{proof}
Suppose that $\alpha \in A(F) \cap \ell A(T_{n+1})$. This means that
$\beta_{1} \in A(T_{n+1})$, and $\delta_{1}(\alpha)(\sigma) = \sigma(\beta_{1})
- \beta_{1}$ gives rise to a cohomology class in $H^{1}(T_{n+1}/F, A[\ell])$.
A refined form of the inflation-restriction sequence gives an exact sequence
\[
\begin{CD}
0 @>>> H^{1}(T_{n}/F, A[\ell]) @>{\rm inf}>> H^{1}(T_{n+1}/F, A[\ell])
@>{\rm res}>> H^{1}(T_{n+1}/T_{n}, A[\ell])^{\mathcal{T}_{n}}.
\end{CD}
\]
Now, $\Gal(T_{n+1}/T_{n})$ acts trivially on $A[\ell]$ and so
\[
  H^{1}(T_{n+1}/T_{n}, A[\ell])^{\mathcal{T}_{n}}
  = \Hom_{\mathcal{T}_{n}}(N^{(n)}/N^{(n+1)}, A[\ell])
  = \Hom_{\mathcal{T}_{1}}(N^{(n)}/N^{(n+1)}, A[\ell]) = 0.
\]
Thus, the inflation map is a bijection between $H^{1}(T_{n+1}/F, A[\ell])$
and $H^{1}(T_{n}/F, A[\ell])$. This implies that the cocycle
$\delta_{1}(\alpha)(\sigma)$ is trivial for $\sigma \in \Gal(T_{n+1}/T_{n})$
and this implies that $\beta_{1} \in A(T_{n})$, as desired.
\end{proof}

\begin{lem}
\label{surjlem3}
Suppose that there is a normal subgroup $H$ of $\mathcal{T}_{1}$
with order coprime to $\ell$ and $A[\ell]^{H} = 0$. Then,
$A(F) \cap \ell A(T_{1}) = \ell A(F)$.
\end{lem}
\begin{proof}
Suppose that $\alpha \in A(F)$ and $\beta_{1} \in A(T_{1})$. This gives
rise to a cocycle $\delta_{1}(\alpha)(\sigma) = \sigma(\beta_{1}) - \beta_{1}$
that represents a cohomology class in $H^{1}(F, A[\ell])$. We have the
inflation-restriction sequence
\[
\begin{CD}
0 @>>> H^{1}(\mathcal{T}_{1}/H, A[\ell]^{H}) @>>> H^{1}(\mathcal{T}_{1}, A[\ell])
@>>> H^{1}(H, A[\ell]).
\end{CD}
\]
Because $A[\ell]^{H} = 0$, the first term is zero, and because
$A[\ell]$ has order a power of $\ell$, which is coprime to $|H|$,
the last term is also zero. Thus, $H^{1}(\mathcal{T}_{1}, A[\ell]) = 0$ by 
exactness. 

We have another inflation-restriction sequence
\[
\begin{CD}
0 @>>> H^{1}(\mathcal{T}_{1}, A[\ell]) @>>> H^{1}(F, A[\ell]) @>{\rm res}>>
H^{1}(T_{1}, A[\ell])
\end{CD}
\]
and since $H^{1}(\mathcal{T}_{1}, A[\ell]) = 0$, it follows that the
restriction map is injective. Since the restriction of $\delta_{1}$ to
$H^{1}(T_{1}, A[\ell])$ is zero, it follows that $\delta_{1}$ is a coboundary
and so $\beta_{1} \in F$.
\end{proof}

The following result gives a convenient method of computing $\f(G)$ in the
case that $\kappa$ is surjective, i.e. $\im \omega \cong \Z_{\ell}^{d} \rtimes
\mathcal{T}$.

\begin{thm}
\label{matrixprop} Suppose that $\kappa$ is surjective. 
Then
\begin{equation}
\label{integ}
  \f(G) = \int_{\mathcal{T}} \ell^{-\ord_{\ell}(\det(M-I))} \, d\mu.
\end{equation}
Here, $d\mu$ denotes the Haar measure on $\mathcal{T}$, normalized
so that $\mu(\mathcal{T}) = 1$, and we take $\ord_{\ell}(0) = \infty$.
\end{thm}

\begin{proof}
We will frequently use the fact that if $X \in M_{d}(\Z_{\ell})$
acts on $V = \Z_{\ell}^{d}$ with $\det(X) \ne 0$, then the image of
$X : V \to V$ has index $\ell^{\ord_{\ell}(\det(X))}$.
Note that if
$\det(M-I) = 0$ then by our convention that $\ord_{\ell}(0) = \infty$ we have
$\ell^{-\ord_{\ell}(\det(M-I))} = 0$.  

Suppose that $\sigma \in \mathcal{G}_{n}$ and $\omega_{n}(\sigma) = (a,M)
\in (\Z/\ell^{n} \Z)^{d} \rtimes \GL_{d}(\Z/\ell^{n} \Z)$. Then, if
$\beta \in U_{n}$, then $\sigma$ fixes $\beta$ if and only if $
  \sigma(\beta) - \beta_{n} = \beta - \beta_{n}.$
Write $\beta = \beta_{n} + \gamma$, where $\gamma \in
A[\ell^{n}]$. Then, $\sigma(\beta) = \sigma(\beta_{n}) + \sigma(\gamma)$
and so
\[
  \sigma(\beta) - \beta_{n} = \sigma(\beta_{n}) - \beta_{n} +
  \sigma(\gamma).
\]
The right hand side equals $\beta - \beta_{n}$ if and only if $
  \sigma(\beta_{n}) - \beta_{n} + \sigma(\gamma) = \gamma.$
If $\omega_{n}(\sigma) = (a,M)$ then this means that $a + M(\gamma) = \gamma,$
whence $(M-I)(-\gamma) = a$. This occurs if and only if $a$ is
in the image of $M-I$.

If $M \in \mathcal{T}_{n}$ with $\det(M-I) \not\equiv 0 \pmod{\ell^{n}}$
and $\tilde{M}$ is any lift of $M$ to $\mathcal{T}$, then
$\ord_{\ell}(\det(\tilde{M}-I)) = \ord_{\ell}(\det(M-I))$ and
therefore the index of the image of $M-I$ (acting on
$(\Z/\ell^{n} \Z)^{d}$) and the index of the image of $(\tilde{M}
- I)$ (acting on $\Z_{\ell}^{d}$) are the same. It follows that
the index of the image of $\det(M-I)$ is
$\ell^{\ord_{\ell}(\det(M-I))}$. Hence, the number of elements of
$\mathcal{G}_{n}$ fixing some point of $U_{n}$
divided by the size of $\mathcal{G}_{n} = \Gal(K_{n}/F)$ is
$$
  \frac{\sum_{M \in \mathcal{T}_{n}} \# \im(M-I)}
{\# \mathcal{T}_{n} \cdot \ell^{dn}}\\
  =
  \frac{\sum^{'} \ell^{dn - \ord_{\ell}(\det(M-I))}}{\# \mathcal{T}_{n}
    \cdot \ell^{dn}} +
  \frac{\sum^{''}
  \# \im(M-I)}{\# \mathcal{T}_{n} \cdot \ell^{dn}},
$$
where $\sum'$ and $\sum^{''}$ are taken over all $M \in \mathcal{T}_n$ with
$\det(M-I) \not\equiv 0 \bmod{\ell^n}$ and $\det(M-I) \equiv 0
\bmod{\ell^n}$, respectively. We may rewrite the first sum as
\[
  \int_{\{ M \in \mathcal{T} : \det(M-I) \not\equiv 0 \pmod{\ell^{n}}\}}
  \ell^{-\ord_{\ell}(\det(M-I))} \, d\mu.
\]
As $n \to \infty$, this
integral tends to
\[
  \int_{\mathcal{T}} \ell^{-\ord_{\ell}(\det(M-I))} \, d\mu
\]
and the second term tends to zero. This establishes \eqref{integ}.
\end{proof}

\section{Tori}
\label{torisec}

The multiplicative group scheme $\G_{m} = {\rm Spec}~\Z[x,y] /
(xy - 1)$ is one of the simplest examples of an algebraic group.
An algebraic torus $A$ of dimension $n$ is an algebraic group
that is isomorphic to $\G_{m}^{n}$ over $F^{\sep}$. If $F$ is a
number field, then there is a bijection between algebraic tori of
dimension $n$ up to $F$-isomorphism and
\[
  H^{1}(\Gal(\overline{F}/F), \Aut_{\overline{F}}(\G_{m}^{n})) \cong \Hom_{{\rm
      cont}}(\Gal(\overline{F}/F), \GL_{n}(\Z)).
\]
In the special case $n = 1$, $\GL_{1}(\Z) \cong \Z/2\Z$ and $\Hom_{{\rm
    cont}}(\Gal(\overline{F}/F), \Z/2\Z) \cong
F^{\times}/(F^{\times})^{2}$. It follows every dimension 1 torus is
isomorphic to one of the form
\[
  x^{2} - dy^{2} = 1
\]
for some $d \in F^{\times}/(F^{\times})^{2}$, where the group law is
given by
\[
  (x_{1}, y_{1}) * (x_{2}, y_{2}) = (x_{1} x_{2} + d y_{1} y_{2},
  x_{1} y_{2} + x_{2} y_{1}).
\]
For such tori, we have the following surjectivity criteria for the 
$\ell$-adic representation $\rho$.

\begin{prop}
\label{torirho} Let $\ell$ be a prime. The $\ell$-adic
representation $\rho : \mathcal{T}_{\infty} \to
\Z_{\ell}^{\times}$ is surjective if and only if the following
conditions are satisfied:
\begin{enumerate}
\item We have $|F(\zeta_{\ell^{3}} + \zeta_{\ell^{3}}^{-1}) : F|
= \frac{\ell^{2} (\ell - 1)}{2}$.
\item If $\ell \equiv 3 \pmod{4}$, then $-\ell d$ is not a square in $F$.
\item If $\ell = 2$, then $-d$ and $-2d$ are not squares in $F$.
\end{enumerate}
\end{prop}
\begin{proof}
Note that the coordinates of the $\ell^{n}$ torsion points on
$A$ are given by
\[
  \left( \frac{\zeta_{\ell^{n}} + \zeta_{\ell^{n}}^{-1}}{2},
  \frac{\zeta_{\ell^{n}} - \zeta_{\ell^{n}}^{-1}}{2 \sqrt{d}}\right).
\]
Assume first that $\ell > 2$. Since the square of
$\frac{\zeta_{\ell^{2}} - \zeta_{\ell^{2}}^{-1}}{2 \sqrt{d}}$ is in
$F(\zeta_{\ell^{2}} + \zeta_{\ell^{2}}^{-1})$, and
\[
  |F(A[\ell^{2}]) : F| =
  |F(A[\ell^{2}]) : F(\zeta_{\ell^{2}} + \zeta_{\ell^{2}}^{-1})|
  |F(\zeta_{\ell^{2}} + \zeta_{\ell^{2}}^{-1}) : F|,
\]
condition (1) above is necessary. The maximal subgroups of
$\Z_{\ell}^{\times}$ are those that contain the kernel of reduction
mod $\ell$, together with the unique subgroup of index $\ell$ in
$\Z_{\ell}^{\times}$. The first condition above rules out the
possibility of the image of $\rho$ landing in this second subgroup,
and so it suffices to determine when the mod $\ell$ Galois representation
is surjective. Let $L = F(\zeta_{\ell}, \sqrt{d})$ and define $\phi :
\Gal(L/F) \to \F_{\ell}^{\times} \times \Z/2\Z$ by $\phi(\sigma) =
(\sigma|_{\mu_{\ell}}, \sigma(\sqrt{d})/\sqrt{d})$. If $\phi$ is
surjective, then there is an element $\sigma \in \Gal(L/F)$ so that
$\sigma(\zeta_{\ell}) = \zeta_{\ell}^{-1}$ and $\sigma(\sqrt{d}) =
\sqrt{d}$. This element $\sigma$ fixes $\zeta_{\ell} +
\zeta_{\ell}^{-1}$, but it sends $\frac{\zeta_{\ell} -
  \zeta_{\ell}^{-1}}{2 \sqrt{d}}$ to its negative.  Thus,
$\frac{\zeta_{\ell} - \zeta_{\ell}^{-1}}{2 \sqrt{d}} \not\in
F(\zeta_{\ell} + \zeta_{\ell}^{-1})$ and so $|F(A[\ell]) :
F(\zeta_{\ell} + \zeta_{\ell}^{-1})| = 2$ and we have that
$|F(A[\ell]) : F| = \ell - 1$ and so $\rho$ is surjective.

Suppose therefore that $\phi$ is not surjective. Condition (1)
implies that $|F(\zeta_{\ell} + \zeta_{\ell}^{-1}) : F| = \frac{\ell - 1}{2}$
and so the image of $\phi$ has index at most 2. There are three
subgroups of $\F_{\ell}^{\times} \times (\Z/2\Z)$ of index 2, and they are
$\{ (a, 1) : a \in \F_{\ell}^{\times} \}$,
$\{ (a^{2}, \pm 1) : a \in \F_{\ell}^{\times} \}$ and
$\{ (a, \legen{a}{\ell}) : a \in \F_{\ell}^{\times} \}$.

In the first case, $\sqrt{d}$ is fixed by $\Gal(L/F)$ and so $\sqrt{d} \in K$.
In this case, $F(A[\ell]) = F(\zeta_{\ell})$, and since
$|F(\zeta_{\ell}) : F| = \ell - 1$, $\rho$ is surjective. 

In the second case, $|F(\zeta_{\ell}) : F| = (\ell - 1)/2$. Here
we have $F(A[\ell]) = F(\zeta_{\ell}, \sqrt{d})$ and this has
degree $\ell - 1$ over $F$ and so $\rho$ is surjective.

In the third case, $\sqrt{d} \in F(\zeta_{\ell})$
and $|F(\zeta_{\ell}) : F| = \ell - 1$. If $\sqrt{d} \in
F(\zeta_{\ell} + \zeta_{\ell}^{-1})$, then $F(A[\ell]) = F(\zeta_{\ell})$
and $\rho$ is surjective. If not, then $\ell \equiv 3 \pmod{4}$,
$-\ell d = \alpha^{2}$ for some $\alpha \in F$, and we have that
\[
  \frac{\zeta_{\ell} - \zeta_{\ell}^{-1}}{2 \sqrt{d}}
  = \frac{\zeta_{\ell} - \zeta_{\ell}^{-1}}{2 \alpha \sqrt{-1/\ell}}
\]
lies in $F(\zeta_{\ell} + \zeta_{\ell}^{-1})$. In this case,
$|F(A[\ell]) : F| = \frac{\ell - 1}{2}$ and $\rho$ is not surjective.

For $\ell = 2$, one computes that $F(A[8]) = F(\sqrt{2}, \sqrt{-d})$.
It follows then that $\rho$ is surjective if and only if
$|F(A[8]) : F| = 4$. This occurs if and only if $2$, $-d$ and $-2d$ are not
squares in $F$. Since $F(\zeta_{8} + \zeta_{8}^{-1}) = F(\sqrt{2})$,
condition (1) guarantees that $2$ is not a square in $F$.
\end{proof}

Next, we have the following surjectivity criteria for the Kummer map
$\kappa$.

\begin{thm}
\label{torikappa} Let $A : x^{2} - dy^{2} = 1$ be a
one-dimensional torus over $F$ and let $\alpha \in A(F)$. Assume that the
$\ell$-adic representation on $A$ is surjective. The Kummer map
$\kappa : \Gal(\conj{F}/T_{\infty}) \to \Z_{\ell}$ is
surjective if and only if the following conditions are satisfied:
\begin{enumerate}
\item $\alpha \not\in \ell A(F)$.
\item If $\ell = 2$, assume that $F(\beta_{1}) \not\subseteq F(A[8])$.
\end{enumerate}
\end{thm}
\begin{proof}
The necessity is clear since either of the above two conditions will
force the image of $\kappa$ to have index a multiple of $\ell$.

First assume $\ell > 2$. The hypotheses of Lemma~\ref{surjlem3} are
satisfied with $H = \mathcal{T}_{1} \cong (\Z/\ell \Z)^{\times}$. Moreover,
$N^{(n)}/N^{(n+1)}$ is one-dimensional with the trivial action,
while $A[\ell]$ has the non-trivial action. Thus, the hypotheses
of Lemma~\ref{surjlem2} are satisfied. Theorem~\ref{surj1} now
implies that $\kappa$ is surjective.

If $\ell = 2$ and $\beta_{1} \in 2 A(T_{n})$ for some $n$,
then Lemma~\ref{surjlem1} implies that 
$F(\beta_{1}) \subseteq F(A[8])$. This contradicts the hypothesis.
Hence, $\beta_{1} \not\in 2 A(T_{n})$ for any $n$, and
Theorem~\ref{surj1} implies that $\kappa$ is surjective. 
\end{proof}

As a consequence we obtain conditions for the surjectivity of $\omega$.

\begin{cor}
\label{tori}
The arboreal representation $\omega : \Gal(K_{\infty}/F) \to
\Z_{\ell} \rtimes \Z_{\ell}^{\times}$ is surjective if and only
if the conditions of Theorem~\ref{torikappa} and
Proposition~\ref{torirho} are satisfied.
\end{cor}
\begin{proof}
It is clear that $\omega$ is surjective if and only if
$\kappa$ and $\rho$ are surjective.
\end{proof}

\begin{ex}
\label{untwistedtorus}
Suppose that $F = \Q$ and $d = 1$. In this case, $x^{2} - y^{2} = 1$
is isomorphic to $\G_{m}$ over $\Q$. If $\ell > 2$ and $\alpha =
(x_{0},y_{0}) \not\in \ell A(\Q)$, then Theorem~\ref{torikappa} and
the above remark demonstrate that $G = \Gal(K_{n}/F) \cong \Z_{\ell}
\rtimes \Z_{\ell}^{\times}$. Moreover, one verifies directly that the
same conclusion holds for $\ell = 2$ as long as the corresponding
point $(\gamma,1/\gamma) = \left(\frac{x_{0} + y_{0}}{2}, 
\frac{x_{0} - y_{0}}{2}\right)$ on $x' y' = 1$ satisfies condition (2) of
Theorem \ref{torikappa}. In this case, $\Q(A[8]) = \Q(\zeta_{8})$
and $\Q(\beta_{1}) = \Q(\sqrt{\gamma})$. Since the quadratic subfields
of $\Q(\zeta_{8})$ are $\Q(\sqrt{2})$, $\Q(i)$, and $\Q(\sqrt{-2})$,
this is equivalent to none of $\pm \gamma$ or $\pm 2 \gamma$ being squares in 
$\Q$.
\end{ex}

\begin{prop}
\label{gmdensity}
Suppose that $\ell$ is prime, and $G \cong \Z_{\ell} \rtimes
\Z_{\ell}^{\times}$.
Then,
\[
  \f(G) = \frac{\ell^{2} - \ell - 1}{\ell^{2} - 1}.
\]
\end{prop}
\begin{proof}
We will apply Theorem~\ref{matrixprop}. Representing elements of
  $\Z_{\ell}^{\times}$ with their $\ell$-adic expansions, we obtain
\begin{equation} \label{zldense}
\mu(\{x \in \Z_{\ell}^{\times} : v_{\ell}(x-1) = n\}) =
\begin{cases} (\ell - 2)/(\ell - 1) & \text{if $n = 0$} \\
1/\ell^{n} & \text{if $n \geq 1$}.
\end{cases}
\end{equation}
The integral in \eqref{integ} is therefore
\begin{equation*}
  \frac{\ell - 2}{\ell - 1} + \sum_{k=1}^{\infty} \frac{1}{\ell^{2k}}
  = \frac{\ell - 2}{\ell - 1} + \frac{1}{\ell^{2} - 1} =
  \frac{\ell^{2} - \ell + 1}{\ell^{2} - 1}.
\end{equation*}
\end{proof}

Returning to Example~\ref{untwistedtorus}, we see that in the case $A
= \G_{m}$, $F = \Q$, we have $\f(G) = (\ell^2 - \ell - 1)/(\ell^2 -
1)$ for general $(\gamma, 1/\gamma) \in \G_m(\Q)$.  More specifically,
if $\gamma \in \Q$ is neither plus or minus a square nor twice a
square, the density of $p$ such that the order of $\gamma \in
(\Z/p\Z)^{\times}$ is odd is $1/3$. Similar results were first proved
by Hasse \cite{hasse}, \cite{hasse2}; see \cite{moree2} for a complete
accounting. For instance, Hasse showed that the density of primes $p$
dividing $2^{n} + 1$ for some $n$ is $\frac{17}{24}$. Note that $p
\mid 2^{n} + 1$ for some $n$ if and only if $2^{n} \equiv -1
\pmod{p}$, that is if and only if $(2,1/2)$ has even order in
$\G_{m}(\F_{p})$. Similarly, Lagarias' result \cite{lagarias} about
primes dividing the $n$th Lucas number $L_{n}$ follows from our
results above since it is easy to see that $p$ divides some $L_{n}$ if
and only if the point $(3/2,1/2)$ on $A : x^{2} - 5y^{2} = 1$ has even
order in $A(\F_{p})$. Finally, \cite{GT}, contains a study of primes
that divide sequences of the shape $a_{1} = t^{2} - t - 2$, $a_{n} =
(a_{n-1} + t)^{2} - (2+t)$. The $n$th term $a_{n}$ comes from $\alpha
= (t,u)$ on $A : x^{2} - dy^{2} = 4$, where $d$ and $u$ are chosen so
that $d$ is squarefree and $t^{2} - 4 = du^{2}$. In particular,
\[
  a_{n} = x([2^{n}](t,u)) - t.
\]
A prime $p$ divides $a_{n}$ if and only if $x([2^{n}](t,u)) \equiv t
\pmod{p}$. This occurs if and only if $(t,u)$ has odd order mod $p$.

\begin{ex}
\label{badtwist}
 Suppose that $F = \Q$, $d = -7$, $\ell = 7$ and $\alpha =
(3/4, 1/4)$. In this case, one can show that $K_{n}$ is the
unique real subfield of $\Q\left(\zeta_{7^{n}}, \left(\frac{3 +
    \sqrt{-7}}{4}\right)^{1/7^{n}}\right)$ and $[F_{n} : F] = 3 \cdot
7^{2n-1}$. The density in this case is $\f(G) = \frac{17}{24}$,
less than the density of $\frac{41}{48}$ that would be obtained if
Proposition~\ref{gmdensity} applied.
\end{ex}

The situation becomes more complex when we consider tori $A$ with $A
\cong \G_{m} \times \G_{m}$ over the algebraic closure of $F$. We will
content ourselves with considering two examples.

\begin{ex}
  Suppose that $F = \Q$, $A \cong \G_{m} \times \G_{m}$ is given by $A
  : xyz = 1$. Let $\ell$, $p$ and $q$ be distinct primes and consider
  multiplication by $\ell$ with $\alpha = (p,q,\frac{1}{pq})$. In this
  case, $F_{n} = \Q(\zeta_{\ell^{n}}, p^{1/\ell^{n}}, q^{1/\ell^{n}})$
  and $\mathcal{G}_{n} \cong (\Z/\ell^{n} \Z)^{2} \rtimes (\Z/\ell^{n}
  \Z)^{\times}$, so that $\omega$ is surjective. One can compute that
  $\f(G) = \frac{\ell^{3} - \ell^{2} - \ell - 1}{\ell^{3} - 1}$.
\end{ex}

\begin{ex}
\label{bigtorus}
 Let $F = \Q$, and let $A$ be defined by
\[
  1 = x^{3} + 2y^{3} + 4z^{3} - 6xyz = N_{\Q(\sqrt[3]{2})/\Q}(x + y
  \sqrt[3]{2} + z \sqrt[3]{4}).
\]
We take $\ell = 2$ and $\alpha = (-1,1,0)$. In this example,
$A \cong \G_{m} \times \G_{m}$ over $L = \Q(\sqrt[3]{2}, \zeta_{3})$.
One can show that
$F_{n} =
\Q(\zeta_{3}, \zeta_{2^{n}}, (\sqrt[3]{2} - 1)^{1/2^{n}}, (\zeta_{3}
\sqrt[3]{2} - 1)^{1/2^{n}})$.
Then,
\[
  \mathcal{G}_{n} \cong (\Z/2^{n} \Z)^{2} \rtimes (S_{3} \times
  (\Z/2^{n} \Z)^{\times}), \text{ and } \f(G) = 67/168.
\]
\end{ex}

\section{Elliptic Curves}
\label{elliptic}

\subsection{Elliptic curves without complex multiplication}

Suppose that $A/F$ is an elliptic curve without complex
multiplication, $F$ is a number field, $\phi = [\ell]$, and $\alpha
\in A(F)$.

To determine the image of $\omega$, we need to determine both the image
of the Kummer map $\kappa : \Gal(K_{\infty}/T_{\infty}) \to T_\ell(A) \cong
\Z_{\ell}^2$ and the image of the associated $\ell$-adic
representation $\rho : \mathcal{T} \to
\GL_{2}(\Z_{\ell})$.

We treat the torsion part first by giving criteria for
$\rho : \mathcal{T} \to \GL_{2}(\Z_{\ell})$ to be
surjective. This problem has been well-studied. In particular, in \cite{serre1}
it is shown that $\rho$ is surjective provided $\ell$ is large enough.

Recall that the $n$-torsion polynomial of an elliptic curve
$E : y^{2} = x^{3} + Ax + B$ is the polynomial whose roots are the
$x$-coordinates of the points of order $n$ in $E(\overline{F})$. 

\begin{prop}
\label{surj2prop}
Let $\ell$ be a prime.  The $\ell$-adic representation
$\rho : \Gal(T_{\infty}/F) \to \GL_{2}(\Z_{\ell})$ is
surjective if and only if the following
conditions hold:
\begin{enumerate}
\item The base field $F$ is linearly disjoint from $\Q(\zeta_{\ell^{n}})$
  for all $n$.
\item $\mathcal{T}_{1} \cong \GL_{2}(\Z/\ell \Z)$.
\item If $\ell = 2$ and $D$ is the discriminant of the 2-torsion
  polynomial, then $-D$, $2D$ and $-2D$ are not squares in $F$, and
  the 4-torsion polynomial is irreducible and its Galois group has
  order 48 over $F$.
\item If $\ell = 3$, then the 9-torsion polynomial is
  irreducible over $F(\zeta_{9})$.
\end{enumerate}
\end{prop}
\begin{proof}
Serre \cite{Serre2} [IV, 3.4, Lemma 3] 
shows that if $\ell \geq 5$, then no proper closed subgroup
of $\SL_{2}(\Z_{\ell})$ surjects onto $\SL_{2}(\Z/\ell \Z)$. From the formula
\[
  \det\rho(\Frob_{p}) = \chi_{\ell}(p),
\]
where $\chi_{\ell}$ is the $\ell$-adic cyclotomic character,
we see that $\rho|_{\Gal(F_{\zeta_{\ell^{\infty}}}/F)}$ surjects onto
$\SL_{2}(\Z_{\ell})$ if and only if the map $\rho|_{F(\zeta_{\ell})}$
surjects onto $\SL_{2}(\Z/\ell \Z)$. The linear disjointness of $F$
with $\Q(\zeta_{\ell^{n}})$ implies that $\det\rho :
\Gal(T_{\infty}/F) \to \Z_{\ell}^{\times}$ is surjective. This
demonstrates that conditions (1) and (2) are necessary for any $\ell$
and sufficient for $\ell \geq 5$.

For $\ell = 3$, the proof of Theorem~\ref{frattini} shows that
$\Phi(\GL_{2}(\Z_{3})) \supseteq N^{(2)}$, where $N^{(k)} = \ker(\GL_{2}(\Z_{\ell})
\to \GL_{2}(\Z/\ell^{k} \Z))$. A computation then shows that there are five
maximal subgroups of $\GL_{2}(\Z/9\Z)$ with indices 2, 3, 3, 4, and 27,
respectively. The maximal subgroups of index 2, 4, and one of those with
index 3 correspond to the maximal subgroups of $\GL_{2}(\Z/3\Z) \cong S_{4}$.
The other maximal subgroup of index 3 is
\[
  \{ M \in \GL_{2}(\Z/9\Z) : \det(M) \equiv \pm 1 \pmod{9} \}.
\]
To ensure the image of $\rho$ does not lie in the maximal subgroup of index
3 described above, it is necessary and sufficient to assume 
that $F$ is linearly disjoint from $\Q(\zeta_{9})$. 
The maximal subgroup of index 27 is generated by
\[
  \left[ \begin{matrix} 0 & 7 \\ 5 & 8 \end{matrix} \right], 
  \left[ \begin{matrix} 2 & 5 \\ 5 & 6 \end{matrix} \right].
\]
Its intersection with $\SL_{2}(\Z/9\Z)$ has order 24. The 9-torsion
polynomial is irreducible over $F(\zeta_{9})$ if and only if $(\im
\rho_{2}) \cap \SL_{2}(\Z/9\Z)$ acts transitively on the
$x$-coordinates of the 9-torsion points. If the image of $\rho$ lies
in the maximal subgroup of index 27, then it cannot act transitively
on the 36 $x$-coordinates, since the order of the group is only 24. On
the other hand, if $\rho_{2}$ is surjective, then since
$\SL_{2}(\Z/9\Z)$ acts transitively on elements of order $9$ in
$(\Z/9\Z)^{2}$, it follows that $\Gal(F(A[9])/F)$ acts transitively on
the elements of order $9$ in $A[9]$, and hence on roots of the
$9$-torsion polynomial. Thus, a necessary and sufficient condition for
$\rho$ to be surjective is that the 9-torsion polynomial is
irreducible over $F(\zeta_{9})$.

For $\ell = 2$, the proof of Theorem~\ref{frattini} shows that
$\Phi(GL_{2}(\Z_{2})) \supseteq N^{(3)}$. A computation then shows that
there are 9 maximal subgroups of $\GL_{2}(\Z/8\Z)$. Seven of these have
index 2, one has index 3, and one has index 4.

To guarantee that the image of $\rho$ does not lie in one of the
subgroups of index 2, it is necessary and sufficient that $F(\sqrt{2})$,
$F(i)$, and the quadratic subfield of $F(A[2])$ are three
independent quadratic extensions of $F$. This is guaranteed by the
condition (1) and the first part of condition (3), 
since $F(\zeta_{8}) \subseteq F(A[8])$. To guarantee that the image of
$\rho$ does not lie in one of the subgroups of index 3,
it is necessary and sufficient that $3 | [F(A[2]) : F]$,
which is guaranteed by condition (2). 

The maximal subgroup of $\GL_{2}(\Z/4\Z)$ of index 4 is generated by
\[
  \left[ \begin{matrix} 0 & 3\\ 3 & 3 \end{matrix}\right],
  \left[ \begin{matrix} 3 & 3\\ 0 & 1 \end{matrix}\right].
\]
Note that if $A$ is written in the form $A : y^{2} = x^{3} + ax + b$
then the action of $[-1]$ on $(x,y)$ is $(x,y) \mapsto (x,-y)$. Hence,
the Galois extension obtained by adjoining the $x$-coordinates
is the image of $\rho_{2}$ in $\GL_{2}(\Z/4\Z) / \langle \pm I \rangle$.
In the case that $\rho_{2}$ is surjective, it has order 48, while in the case
where the image of $\rho_{2}$ is contained in the maximal subgroup of
index 4 (and all the other conditions of the theorem are met) it has order 12.
\end{proof}

\begin{rem}
There are modular curves of genus zero that parametrize the elliptic curves
$E/\Q$ for which the mod $\ell$ representations are surjective, but the
mod $\ell^{2}$ representations are not. In \cite{Elkies}, Elkies computes this
parametrization for $\ell = 3$ and gives the first examples of curves $E/\Q$
for which the mod $3$ representation is surjective but the mod 9 representation
is not. The CM elliptic curve $y^{2} + y = x^{3}$ is the smallest conductor
curve for which the mod 2 representation is surjective, but the mod 4
representation is not.
\end{rem}

\begin{thm}
\label{surj2}
Suppose that the $\ell$-adic representation $\rho : \Gal(T_{\infty}/F)
\to \GL_{2}(\Z_{\ell})$ is surjective.  Then the Kummer map $\kappa :
\Gal(K_{\infty}/T_{\infty}) \to \Z_{\ell}^2$ is surjective if and only
if the following conditions hold:
\begin{enumerate}
\item The point $\alpha \not\in \ell A(F)$.
\item If $\ell = 2$, $F(\beta_{1}) \not\subseteq F(A[4])$. 
\end{enumerate}
\end{thm}
\begin{proof}
It is clear that if either of the stated conditions fails to hold,
then $\kappa$ fails to be surjective.

Assume that $\ell > 2$ and $\alpha \not\in \ell A(F)$. Then,
$H = Z(\mathcal{T}_{1}) \cong (\Z/\ell \Z)^{\times}$ is
a normal subgroup of $\mathcal{T}_{1}$ with order coprime to $\ell$
and with $A[\ell]^{H} = 0$. Thus, Lemma~\ref{surjlem3} implies that
$A(F) \cap \ell A(T_{1}) = \ell A(F)$. Next, $N^{(n)}/N^{(n+1)}$ is isomorphic to
$M_{2}(\F_{\ell})$ as a $\mathcal{T}_{1}$-module with the conjugation action.
This decomposes as a direct sum of a three-dimensional and a one-dimensional
$\mathcal{T}_{1}$-module, and hence $\Hom_{\mathcal{T}_{1}}(N^{(n)}/N^{(n+1)},
A[\ell]) = 0$. Thus, Lemma~\ref{surjlem2} implies that
$A(F) \cap \ell A(T_{n}) = A(F) \cap \ell A(T_{n+1})$ for all $n \geq 1$.
It follows that for any $n \geq 1$, $A(F) \cap \ell A(T_{n}) = \ell A(F)$.
Finally, the surjectivity of $\rho$ implies that $A[\ell]$ is irreducible
as a $\mathcal{T}_{1}$-module. Thus, Theorem~\ref{surj1} implies that
$\im \omega_{n} \cong A[\ell^{n}] \rtimes \GL_{2}(\Z/\ell^{n} \Z)$
for all $n \geq 1$.

When $\ell = 2$, again we have that $A[\ell]$ is irreducible as a
$\mathcal{T}_{1}$-module. Thus, Lemma~\ref{surjlem1} implies that
if $F(\beta_{1}) \subseteq T_{n}$, then $\Gal(T_{n}/F(\beta_{1}))$ is
a maximal subgroup of $\mathcal{T}_{n}$. The condition that $F(\beta_{1})
\not\subseteq F(A[2])$ implies that $\Gal(K_{1}/F) \cong (\Z/2\Z)^{2}
\rtimes \GL_{2}(\Z/2\Z)$ and that $|F(\beta_{1}) : F| = 4$. 
Since the only maximal subgroup of index 4 of $\GL_{2}(\Z_{2})$ 
contains $N^{(2)} = \ker \rho_{2}$, it follows that if 
$F(\beta_{1}) \subseteq T_{n}$ for some $n$, then $F(\beta_{1}) \subseteq
T_{2}$. This contradicts the hypotheses of the theorem. Thus,
$F(\beta_{1}) \not\subseteq T_{n}$ for any $n$ and Theorem~\ref{surj1}
gives $\im \omega_{n} \cong A[\ell^{n}] \rtimes \mathcal{T}_{n}$ for
all $n \geq 1$.
\end{proof}

\begin{rem}
The condition that $F(\beta_{1}) \not\subseteq F(A[4])$ is necessary. 
In particular, if $A : y^{2} + y = x^{3} - 3x + 4$ and $\alpha = (4,7)$,
then $\rho$ is surjective, but $F(\beta_{1}) \subseteq F(A[4])$.
\end{rem}

\begin{cor} \label{surj2cor} The arboreal representation $\omega :
  \Gal(K_{\infty}/F) \to (\Z_\ell)^{2} \rtimes \GL_{2}(\Z_\ell)$ is
  surjective if and only if the conditions of Theorem \ref{surj2} and
  Proposition \ref{surj2prop} are satisfied.
\end{cor}
\begin{proof}
The necessity is clear. The sufficiency follows from the basic fact that
if $N \normal G$ and $M \subseteq G$ is a subgroup with
$M \cap N = N$ and $M/N = G/N$, then $M = G$.
\end{proof}

\begin{ex}
\label{noncmex}
Let $A : y^{2} + y = x^{3} - x$. Then $A$ is an elliptic
curve of conductor 37. In \cite{serre1} (pg. 310, 5.5.6),
it is shown that $\Gal(\Q(A[\ell])/\Q) \cong \GL_{2}(\Z/ \ell \Z)$ for
all $\ell$. It is also known that $\alpha = (0,0)$ is a generator
of $E(\Q) \cong \Z$. One can check that the 9-torsion polynomial
is irreducible over $\Q(\zeta_{9})$, and this implies that the $\omega$
representation is surjective for $\ell > 2$. For $\ell = 2$,
one can check that the 4-torsion polynomial has Galois group of
order 48, and that the discriminant of the two-torsion polynomial is
$592$. Further, Frobenius at 19 acts on $K_{1}$ with order 4,
while it acts on $\Q(A[4])$ with order 2. It follows that $K_{1} \not\subseteq
\Q(A[4])$, and hence the $\omega$ representation is surjective for all $\ell$.
\end{ex}

Now we turn to the problem of computing the density $\f(G)$ in the
situation that $\omega$ is surjective, i.e. $\Gal(K_{\infty}/F) \cong
(\Z_\ell)^{2} \rtimes \GL_{2}(\Z_\ell)$.  

\begin{thm}
\label{gl2den}
If $|\cdot|_{\ell}$ is the normalized absolute value on $\Z_{\ell}$,
then we have
\[
  \int_{\GL_{2}(\Z_{\ell})} |\det(M-I)|_{\ell}^{-1} \, d\mu
= \frac{\ell^{5} - \ell^{4} - \ell^{3} + \ell + 1}{\ell^{5} - \ell^{3}
- \ell^{2} + 1}.
\]
\end{thm}
\begin{proof}
It is necessary to count the number $c_{n}$ of matrices $M \in
\GL_{2}(\Z/\ell^{n} \Z)$ with $\det(M-I) \equiv 0 \pmod{\ell^{n-1}}$ but
$\det(M-I) \not\equiv 0 \pmod{\ell^{n}}$. Then the desired integral is
\[
  \sum_{n=1}^{\infty} \frac{c_{n}}{\ell^{n-1} \# \GL_{2}(\Z/\ell^{n} \Z)}.
\]

First, we compute $c_{1}$. This is the number of matrices $M \in
\GL_{2}(\F_{\ell})$ so that $M-I$ is invertible, that is, 1 is not an
eigenvalue of $M$. We will first count the number of matrices in
$\GL_{2}(\F_{\ell})$ that do have 1 as an eigenvalue. This implies that the
other eigenvalue is in $\F_{\ell}$ and hence $M$ has a Jordan form
over $\F_{\ell}$. It follows that $M$ is similar to one of
\[
  \left[ \begin{matrix} 1 & 0\\ 0 & \lambda \end{matrix} \right],
  \lambda \ne 1, \text{ or } \left[ \begin{matrix} 1 & 1\\ 0 & 1
    \end{matrix} \right], \text{ or }
\left[ \begin{matrix} 1 & 0\\ 0 & 1 \end{matrix} \right].
\]
The size of the conjugacy class is the index of the centralizer. We
can easily compute that the centralizer of the first matrix is
$\left\{ \left[ \begin{matrix} a & 0 \\ 0 & b \end{matrix} \right]
  \right\}$ which has size $(\ell - 1)^{2}$. The centralizer of the
  second matrix is
$\left\{ \left[ \begin{matrix} a & b\\ 0 & a \end{matrix} \right] \right\}$
which has size $\ell (\ell - 1)$, and the centralizer of the third
matrix is $\GL_{2}(\F_{\ell})$, which has order $(\ell^{2} -
1)(\ell^{2} - \ell)$. It follows that
\[
  c_{1} = \# \GL_{2}(\F_{\ell}) - \ell (\ell + 1) (\ell - 2) - (\ell - 1)
  (\ell + 1) - 1 = \ell^{4} - 2 \ell^{3} - \ell^{2} + 3 \ell.
\]

For $n \geq 2$, we pick a matrix $M = \left[ \begin{matrix} a & b \\ c & d \end{matrix}\right] \in \GL_{2}(\F_{\ell})$ and count
how many $\tilde{M} \in \GL_{2}(\Z/\ell^{n} \Z)$ there are with
$\tilde{M} \equiv M$ and $\det(\tilde{M} - I) \equiv 0
\pmod{\ell^{n-1}}$ but $\det(\tilde{M} - I) \not\equiv 0
\pmod{\ell^{n}}$. Write 
\[
  \tilde{M} - I = \left[ \begin{matrix} \alpha & \beta \\ \gamma & \delta
  \end{matrix} \right].
\]
The condition that $\det(\tilde{M} - I) \equiv 0 \pmod{\ell^{n-1}}$
but $\det(\tilde{M} - I) \not\equiv 0 \pmod{\ell^{n}}$ is equivalent to
the existence of $i \in (\Z/\ell^{n} \Z)$ and 
$\epsilon \in (\Z/\ell \Z)^{\times}$ so that
\begin{align*}
  & \alpha \delta \equiv i + \epsilon \ell^{n-1} \pmod{\ell^{n}}\\
  & \beta \gamma \equiv i \pmod{\ell^{n}}.\\
\end{align*}
Hence, the number of such $M$ is
\begin{align*}
  & \sum_{\epsilon=1}^{\ell - 1} \sum_{i=0}^{\ell^{n} - 1}
  \# \{ (\alpha, \delta) : \alpha \delta \equiv i + \epsilon \ell^{n-1}
  \pmod{\ell^{n}}, \alpha \equiv a - 1 \pmod{\ell}, \delta \equiv
  d - 1 \pmod{\ell} \}\\ 
  & \cdot \# \{ (\beta, \gamma) : \beta \gamma \equiv i \pmod{\ell^{n}},
  \beta \equiv b \pmod{\ell}, \gamma \equiv c \pmod{\ell} \}.
\end{align*}
We use the following simple lemma to compute the quantities that appear
in the above expression. We omit the proof of the lemma.

\begin{lem}
\label{count}
Suppose that $a, b \in \Z/\ell \Z$, $c \in \Z/\ell^{n} \Z$, and $n
\geq 2$. Then,
the number of pairs $(\alpha, \beta) \in (\Z/\ell^{n} \Z)$ with
$\alpha \beta \equiv c \pmod{\ell^{n}}$ with $\alpha \equiv a
\pmod{\ell}$ and $\beta \equiv b \pmod{\ell}$ is
\[
\begin{cases}
  0 & ab \not\equiv c \pmod{\ell}\\
  \ell^{n-1} & ab \equiv c \pmod{\ell} \text{ and one of } a \text{ or
  } b \text{ is nonzero}.\\
  (\ell - 1)(\ord_{\ell}(c) - 1) \ell^{n-1}
  & a \equiv b \equiv c \equiv 0 \pmod{\ell},
  c \not\equiv 0 \pmod{\ell^{n}}\\
  (n \ell - n - \ell + 2) \ell^{n-1} & a \equiv b \equiv c \equiv 0
  \pmod{\ell}, c \equiv 0 \pmod{\ell^{n}}.\\
\end{cases}
\]
\end{lem}

If $M \not\equiv I \pmod{\ell}$ but $M$ has one as an eigenvalue, a
straightforward computation using Lemma~\ref{count} shows that there
are $(\ell - 1) \ell^{3n - 3}$ matrices $\tilde{M} \in
\GL_{2}(\Z/\ell^{n} \Z)$ with $\ord_{\ell}(\det(\tilde{M} - I)) = n-1$
for each $M \in \GL_{2}(\F_{\ell})$. There are $\ell^{3} - 2 \ell - 1$
matrices that fall into this case.

If $M \equiv I \pmod{\ell}$, a more lengthy computation using
Lemma~\ref{count} shows that there are
\[
  (\ell^{2} - 1) \ell^{3n - 3} - (\ell^{2} - 1) \ell^{2n-1}
\]
matrices $\tilde{M}$ in $\GL_{2}(\Z/\ell^{n} \Z)$ with
$\ord_{\ell}(\det(\tilde{M} - I)) = n-1$. Hence, we have
\[
  c_{n} = (\ell - 1)^{2} (\ell + 1) \ell^{3n-2} - (\ell^{2} - 1) \ell^{2n-1}.
\]
Hence, we may split up
\[
  \sum_{n=1}^{\infty} \frac{c_{n}}{\ell^{n-1} \# \GL_{2}(\Z/\ell^{n} \Z)}
\]
as a sum of two geometric series, and we get
\[
\mathcal{F}(G) = \frac{\ell^{5} - \ell^{4} - \ell^{3} + \ell +
  1}{\ell^{5} - \ell^{3} - \ell^{2} + 1}.
\]
\end{proof}

\subsection{Complex Multiplication} \label{cm}

Suppose that $A$ is an elliptic curve defined over a number field $F$,
and that $A$ has complex multiplication.  Then ${\rm
  End}_{\overline{F}}(A) \cong R$, where $R$ is an order in an
imaginary quadratic field $L$.  Suppose first that $L \subseteq F$,
and put $R_\ell = R \otimes \Z_\ell$. Let $\mathcal{T} =
\Gal(T_{\infty}/F)$, so that the action of $\mathcal{T}$ on $A[\ell^{\infty}]$ gives
the $\ell$-adic Galois representation associated to $A$.  Then
$\mathcal{T}$ is known to be isomorphic to a subgroup of $R_\ell^{\times}$, provided that $\ell$ does not ramify in $L$ or divide the index of $R$ in the maximal order of $L$
(see e.g. \cite[p.~502]{serre-tate}).  We also have the
analogue of Serre's open image theorem, namely that for any $\ell$,
$\mathcal{T}$ must have finite index in $R_\ell^{\times}$ and in fact
$\mathcal{T} \cong R_\ell^{\times}$ for all but finitely many $\ell$ \cite[p.~302]{serre1}. \label{serre}

A subgroup of $GL_2(\Z_\ell)$ that is isomorphic to $R_\ell^{\times}$
is called a {\em Cartan subgroup}, which we denote by $C$.  In the
case where $L \not\subseteq F$, we have that $\mathcal{T}$ is a
subgroup of the normalizer $N$ of some Cartan subgroup $C$, which contains $C$ as a subgroup
of index two.  Indeed,
$[\mathcal{T} : \mathcal{T} \cap C] = [L : F \cap L] = 2$, and thus
$\mathcal{T}$ is the normalizer of its image in $C$.

We begin by addressing the image of $\rho$.  
\begin{prop} \label{cmmainprop} Let $A$ be an elliptic curve defined
  over a number field $F$, and suppose that the image of $\rho:
  \mathcal{T} \to GL_2(\Z_\ell)$ is contained in the normalizer $N$ of
  a Cartan subgroup but not in a Cartan subgroup.  Denote by $N_m$ the
  image of $N$ in $GL_2(\Z/\ell^m \Z)$.  If $\ell \geq 3$, then $\rho$
  maps onto $N$ if and only if $\mathcal{T}_{2} \cong N_{2}$.  For
  $\ell = 2$, the same conclusion holds if and only if
  $\mathcal{T}_{3} \cong N_{3}$
\end{prop}

\begin{rem} Proposition \ref{cmmainprop} also holds in the case where
  the image of $\rho$ is contained in a Cartan subgroup $C$, with
  analogous conditions ensuring that $\rho$ maps onto $C$.
\end{rem}

\begin{proof} The only if direction is trivial.  Let $C$ be the Cartan
  subgroup of $N$, and suppose that $\mathcal{T}_{2} \cong N_{2}$
  ($\mathcal{T}_{3} \cong N_{3}$ for $\ell = 2$).  Denote by $C_m$ the
  image of $C$ in $GL_2(\Z/\ell^m \Z)$, and recall $C \cong (R \otimes
  \Z_\ell)^\times$, where $R$ is an order in an imaginary quadratic number
  field. Thus $\mathcal{T} \cap C$ surjects onto $C_2$ ($C_3$ if $\ell
  = 2$).  We will show that this implies $\mathcal{T} \cap C$ surjects
  onto $C/\Phi(C)$, where $\Phi(C)$ is the Frattini subgroup of $C$.
  It follows that $\mathcal{T} \cap C = C$, and since $\mathcal{T}$ is
  not contained in $C$ this shows $\mathcal{T} = N$.
  
  To determine $\Phi(C)$, first note that if $S$ is the valuation ring
  in an unramified extension of $\Q_{\ell}$ of degree $d$, then the
  $\ell$-adic logarithm gives an isomorphism $S^{\times} \cong
  \F_{\ell^d}^{\times} \times S$ if $\ell \geq 3$ and $S^{\times}
  \cong \Z/2\Z \times \F_{\ell^d}^{\times} \times S$ if $\ell = 2$,
  where $\F_{\ell^d}$ is the finite field with $\ell^d$ elements
  \cite[p.~257]{robert}.  Since $\ell S$ is the Frattini
  subgroup of $S$, it follows that the log of any maximal subgroup of
  $S^{\times}$ must contain $\ell S$, whence $\log \Phi(S^{\times})
  \supseteq \ell S$.  Under the log isomorphism, $\ell S$ corresponds
  to $\{x \in S^{\times} : x \equiv 1 \bmod{\ell^2}\}$ if $\ell \geq
  3$ and $\{x \in S^{\times} : x \equiv 1 \bmod{\ell^3}\}$ if $\ell =
  2$.  Thus if $G \leq S^{\times}$ and $G$ has full image in
  $(S/\ell^2S)^{\times}$ ($(S/\ell^3S)^{\times}$ if $\ell = 2$) then
  $G$ surjects onto $S^{\times}/\Phi(S^{\times})$ and hence $G =
  S^{\times}$.

  If $\ell$ is inert in $R$, then $R_{\ell}$ is isomorphic to the
  valuation ring in an unramified quadratic extension of $\Q_{\ell}$,
  and the result is proved by the previous paragraph.  If $\ell$ splits
  in $R$, then $R_{\ell}^{\times} \cong \Z_{\ell}^{\times} \times
  \Z_{\ell}^{\times}$, and we have $\log \Phi(R_{\ell}^{\times})
  \supseteq \ell\Z \times \ell\Z$.  The proof then follows as in the
  previous paragraph.
\end{proof}

\begin{thm} \label{cmmain} Let $A$ be an elliptic curve defined over a
  number field $F$, and suppose that the image of $\rho: \mathcal{T}
  \to GL_2(\Z_\ell)$ is the full normalizer $N$ of a Cartan subgroup.  Suppose further that we are not in the case where $\ell = 2$ and the underlying Cartan subgroup is
  split.  Then the Kummer map $\kappa : \Gal(K_{\infty}/T_{\infty}) \to
  \Z_{\ell}^2$ is surjective if and only if $\alpha \not\in \ell
  A(F)$.
  \end{thm}

\begin{proof} The only if direction is trivial.  For the other
  direction, assume first that $N$ is the normalizer of a Cartan
  subgroup $C$, excluding the case where $C$ is split and $\ell = 2$.
  We apply Theorem \ref{surj1} with $m=1$.  The first hypothesis of
  Theorem \ref{surj1} is satisfied since a computation shows that
  $\mathcal{T}_1$ acts irreducibly on $A[\ell]$ (indeed, transitively
  when $C$ is non-split).

  To verify the second hypothesis of Theorem \ref{surj1} with $m=1$,
  we first apply Lemma \ref{surjlem3} with $H = C_1$, the reduction
  modulo $\ell$ of $C$.  This works since $C_1$ has order $\ell^2 -
  1$ in the non-split case and $(\ell - 1)^2$ with $\ell > 2$ in the
  split case, and clearly $A[\ell]^{C_1} = 0$.  We may also apply
  Lemma \ref{surjlem2} for all $n \geq 1$, since the two-dimensional
  $\mathcal{T}_1$-module $N^{(n)}/N^{(n+1)}$ has a one-dimensional
  submodule (namely that generated by the multiplicative identity
  matrix), while the two-dimensional $\mathcal{T}_1$-module $A[\ell]$
  is irreducible.  Theorem \ref{surj1} now applies to prove the
  theorem.
\end{proof}

\begin{rem}
In the setup of Theorem \ref{cmmain}, when $\ell = 2$ and the underlying Cartan subgroup is split, $\mathcal{T}_m$ does not act irreducibly on $A[\ell^m]/A[\ell^{m-1}]$ for any $m$, meaning we cannot apply Theorem \ref{surj1}.  However, we can obtain the conclusion of Theorem \ref{cmmain} under the
stronger assumption that $[T_{3}(\beta_{1}) : T_{3}] = 4$.  Indeed, a computation of the Frattini subgroup 
of $\Z_2^2$ shows that if $\kappa$ is not surjective then $[T_{n}(\beta_{1}) : T_{n}] \leq 2$ for some $n$.  This implies that $T_1(\beta_1) \cap T_n$ contains a degree-two (and therefore minimal) subextension of $T_{\infty}/T_1$.  It follows from the proof of Proposition \ref{cmmainprop} that such an extension lies in $T_3$, and one deduces $[T_{3}(\beta_{1}) : T_{3}] \leq 2$.
\end{rem}

The following corollary has the same proof as Corollary \ref{surj2cor}.

\begin{cor} \label{cmmaincor} Let $N$ be as in Theorem \ref{cmmain},
  and let $\ell \geq 3$.  The arboreal representation $\omega :
  \Gal(K_{\infty}/K) \to (\Z_\ell)^{2} \rtimes N$ is surjective if and
  only if the conditions of Theorem \ref{cmmain} and Proposition
  \ref{cmmainprop} are satisfied.  When $\ell = 2$ the conditions of
  the above remark are equivalent to the surjectivity of $\omega$.
\end{cor}

Now we compute the densities $\f(G)$ in the CM case.

\begin{thm} \label{cmcomp}
Let $C$ be a Cartan subgroup of $\GL_2(\Z_{\ell})$, and let $G =
\Z_{\ell}^2 \rtimes C$ with the natural action.  Let $h(x) = (x^2
- x - 1)/(x^2 - 1)$.  Then $\f(G) = h(\ell)^{2}$ if $C$ is split
and $h(\ell^2)$ if $C$ is inert. If $G = \Z_{\ell}^2 \rtimes N$,
where $N$ is the normalizer of a Cartan subgroup, then $\f(G) =
(h(\ell)^2 + h(\ell))/2$ in the split case and $(h(\ell^2) +
h(\ell))/2$ in the inert case.
\end{thm}
\begin{proof} Let $\mu$ be the Haar measure, and suppose first
that $C$ is not-split, whence $C \cong R_{\ell}^{\times}$, where
  $R_{\ell}$ may be taken to be the valuation ring in an unramified
  quadratic extension of $\Q_{\ell}$.  By Theorem~\ref{matrixprop}, to
  find $\f(G)$ it is enough to compute $t_n := \mu(\{x \in
  R_{\ell}^{\times} : v_{\ell}(x-1) = n\})$ for each $n \geq 0$ and then
  evaluate the integral in \eqref{integ}.  Since $\ell$ is a
  uniformizer for $R_{\ell}$ and the residue field has order $\ell^2$,
  we have $t_0 = (\ell^2-2)/(\ell^2-1)$.  When $n \geq 1$, for $x-1$
  to have valuation precisely $n$ its $\ell$-adic expansion must have
  constant term 1, order-$i$ term $0$ for $1 \leq i \leq n-1$, and
  order-$n$ term non-zero.  Thus for $n \geq 1$, $t_n =
  1/(\ell^{2}-1) \cdot 1/\ell^{2(n-1)} \cdot (\ell^2 - 1)/\ell^2 =
  1/\ell^{2n}$.  The integral in \eqref{integ} is therefore
$$\frac{\ell^2 - 2}{\ell^2-1} + \sum_{n=1}^{\infty} \frac{1}{\ell^{4n}} = \frac{\ell^4 - \ell^2 - 1}{\ell^4 - 1},$$
and this last expression is just $h(\ell^2)$.

Now suppose that $C$ is split, whence $C \cong \Z_{\ell}^{\times}
\times \Z_{\ell}^{\times}$.  In this case the Haar measure on $C$ is
just the product of the Haar measure $\mu$ on each copy of
$\Z_{\ell}^{\times}$. The expression for $\mu(\{x \in
\Z_{\ell}^{\times} \times \Z_{\ell}^{\times} : v_{\ell}(x-1) = n\})$
thus has $n+1$ terms, since the valuations of the two coordinates of
$x-1$ must sum to $n$.  From \eqref{zldense} it follows that for $n
= 0$ we get a measure of $(\ell-2)^2/(\ell-1)^2$, while for $n \geq
1$ a short computation shows the measure is
$$\frac{1}{\ell^n} \left( 2\cdot \frac{\ell-2}{\ell-1} + n-1 \right).$$
The integral in \eqref{integ} thus becomes
$$\frac{(\ell-2)^2}{(\ell-1)^2} \; + \; \frac{2\ell - 4}{\ell-1} \sum_{n=1}^{\infty} \frac{1}{\ell^{2n}} \; + \;
\sum_{n=1}^{\infty} \frac{n-1}{\ell^{2n}},$$ and after evaluation of
these sums one obtains $(\ell^4 - 2\ell^3 - \ell^2 + 2\ell +
1)/(\ell^2-1)^2$, which is equal to $h(\ell)^2$.

We now consider the case $G = \Z_{\ell}^2 \rtimes N$, where $N$
is the normalizer of a Cartan subgroup.  We have $[N:G] =2$, and
thus we need only determine the integral in \eqref{integ} on the
non-identity coset of $C$ in $N$.   When $C$ is non-split, let
$\gamma \in R_{\ell}$ be such that $R_{\ell} =
\Z_{\ell}[\gamma]$ with $x^2 + cx + d$ the minimal polynomial of
$\gamma$.  Note that $\ord_{\ell}(c) = 0$.   We thus have in the
split and non-split cases, respectively, that the non-identity
coset of $C$ in $N$ consists of all
$$M = \left[ \begin{array}{cc}
0 & a  \\
b & 0 \end{array} \right], \qquad M = \left[ \begin{array}{cc}
a & bd - ac  \\
b & -a  \end{array} \right].$$ In the former case we have
$\det(M-I) = 1 - ab$ and in the latter $\det(M-I) = 1 - (a^2 -
abc + db^2)$.  The maps $(a,b) \mapsto ab$ and $a + b \gamma
\mapsto a^2 - abc + db^2$ define homomorphisms $\phi_1$ and
$\phi_2$ mapping $R_{\ell}^{\times} \rightarrow
\Z_{\ell}^{\times}$ in the respective cases ($\phi_2$ is the norm
homomorphism).  Both $\phi_1$ and $\phi_2$ are surjective for
$\ell \geq 3$, as their images properly contain the squares in
$\Z_{\ell}^{\times}$.  For $\ell = 2$ the surjectivity of
$\phi_1$ is clear, while for $\phi_2$ it is useful to take
$\gamma = \zeta_3$, so that $c = d = 1$.  Then $\im \phi_2$
contains the squares and is surjective on $(\Z/8\Z)^{\times}$,
and thus is surjective.  The sets $\{x \in R_{\ell} :
\ord_{\ell}(1 - \phi_i(x)) = n\}$ all have the form
$\phi_i^{-1}(S)$, where $S$ is defined via congruence conditions
modulo $\ell^{n+1}$.  Since the $\phi_i$-preimage of any
congruence class modulo $\ell^{n+1}$ contains the same number of
classes, it follows that $\mu(\phi_i^{-1}(S)) = \mu(S)$, where
the first measure is the Haar measure on $R_{\ell}^{\times}$ and
the second is that on $\Z_{\ell}^{\times}$.  Therefore finding
the integral in \eqref{integ} reduces to the same computation as
in Theorem \ref{gmdensity}, which comes to $h(\ell)$.
\end{proof}

\begin{ex}
\label{cm5split}
Let $F = \Q$, $A : y^{2} = x^{3} + 3x$, $\alpha = (1,-2)$ and $\ell =
5$. The elliptic curve $A$ has CM by the full ring of integers $\Z[i]$
in $L = \Q(i)$, and $5$ splits in $\Z[i]$. One can compute the
Mordell-Weil group $A(\Q)$ and check that $\alpha$ is a
generator. Hence $\alpha \not\in \ell A(\Q)$.  Next, we will show that
$\Gal(\Q(A[25])/\Q) \cong N_{2}$, which has order 800. If $\lambda$ is
a prime ideal above $5$, one can explicitly construct a point $P \in
A[\lambda]$ that lies in a degree 4 extension of $L$. This shows that
the natural homomorphism $\Gal(L(A[\lambda])/L) \to
(\Z[i]/\lambda)^{\times}$ is an isomorphism, and therefore
$\Gal(L(A[\lambda])/L)$ is cyclic of order 4. Moreover, the quadratic
subfield of $L(A[\lambda])/L$ is ramified at $\lambda$.

Explicit class field theory (see Theorem 2.5.6 of \cite{BigSilverman})
shows that the extension obtained by adjoining the squares of the
$x$-coordinates of $A[\lambda^{2}]$ to $L$ has degree 5. Let $M_{1}$
be the compositum of the extension obtained by adjoining the squares
of the $x$-coordinates of $A[\lambda^{2}]$ and all coordinates of the
points in $A[\lambda]$. From above, we have $|M_{1} : L| = 20$, and
that every subextension of $M_{1}$ is ramified at $\lambda$. Let
$\overline{\lambda}$ be the other prime above 5, and let $M_{2}$ be
the extension obtained by adjoing all coordinates of points in
$A[\overline{\lambda}]$, and the squares of the $x$-coordinates of
points in $A[\overline{\lambda}^{2}]$. Similarly, $|M_{2} : L| = 
20$ and every subextension of $M_{2}$ is ramified at $\overline{\lambda}$.

Since $\Z[i]$ has class number one, $L$ has no unramified abelian extensions
and hence $M_{1} \cap M_{2} = L$ and $|M_{1} M_{2} : L| = 400$.
Now, $M_{1}, M_{2} \subseteq L(A[25])$, and
the natural map $\Gal(L(A[25])/L) \to (\Z[i]/25\Z[i])^{\times}$ is injective.
Since $|(\Z[i]/25\Z[i])^{\times}| = 400$, it follows that the above map is
surjective, and $M_{1} M_{2} = L(A[25])$. Finally,
since $\Q(A[25])$ is generalized dihedral over $\Q$, it contains $L$
and hence $|\Q(A[25]) : \Q| = 800$, as desired.

Thus the hypotheses of Theorem \ref{cmmain} are satisfied, and we
conclude by Theorem \ref{cmcomp} and Theorem \ref{interp} that
$\overline{\alpha}$ has order prime to $5$ for $((19/24)^2 + 19/24)/2
= 817/1152 \approx 0.71$ of primes $p$.  Compare this to the generic
value of $2381/2976 \approx 0.80$ in the non-CM case.
\end{ex}

\begin{ex}
\label{cmnonsplit}
Let $K = \Q$, $A : y^{2} = x^{3} + 3$, $\alpha = (1,2)$ and $\ell =
2$.  The elliptic curve $A$ has CM by $\Z[\zeta_3]$, $\alpha$ is a
generator of the Mordell-Weil group of $A$, and $2$ is inert in
$\Z[\zeta_3]$. We will show that $\Gal(\Q(A[8])/\Q) \cong N_{3}$,
which has order 96. It is easy to see that $\Q(A[2]) = \Q(\zeta_{3},
(-3)^{1/3})$. Thus, $3$ divides $|\Q(A[8]) : \Q|$ and $L \subseteq \Q(A[8])$, where
$L = \Q(\zeta_3)$.
Explicit class field theory predicts that the extension $M$ of $L$
obtained by adjoining the cubes of the $x$-coordinates of points in
$A[8]$ has degree $8$. Further, this extension is only ramified at
$2$, and hence every subextension of $M$ is ramified at $2$ since $L$ has
no unramified abelian extensions. 

In addition, the 4-torsion polynomial is $x^{6} + 60x^{3} - 72$.
Therefore if $\alpha$ is the cube of the $x$-coordinate of
a 4-torsion point, then $\alpha^{2} + 60\alpha - 72 = 0$. Therefore,
the $y$-coordinate $\beta$ of a 4-torsion point satisfies $\beta^{2} 
= \alpha + 3$ and so
\[
  (\beta^{2} - 3)^{2} + 60(\beta^{2} - 3) - 72 = \beta^{4} + 54\beta^{2} 
- 243 = 0.
\]
The discriminant of the polynomial $x^{4} + 54x^{2} - 243$ 
is $-2^{12} \cdot 3^{15}$, which is a a square in $L$.
It follows that $L(\beta)/L$ is a Klein-4 extension,
and is given by $L(\beta) = L(i, \sqrt{-1 + \zeta_{3}})$.
The extension $L(\sqrt{-1 + \zeta_{3}})/L$ is ramified at the prime ideal
above $3$ in $O_{L}$ and hence is not contained in $M$.
It follows that $16$ divides $[L(A[8]) : L]$ and hence
$32$ divides $[\Q(A[8]) : \Q] = 2 [L(A[8]) : L]$. Thus,
$[\Q(A[8]) : \Q] = 96$, as desired.

Since $\alpha \not\in 2 A(\Q)$, $\omega$ is surjective.  By Theorem
\ref{cmcomp} and Theorem \ref{interp} we conclude that
$\overline{\alpha}$ has odd order for $(11/15 + 1/3)/2 = 8/15 \approx
0.533$ of primes $p$.
\end{ex}

\begin{ex}
\label{cmsplit}
Let $K = \Q$, $A : y^{2} = x^{3} - 207515x + 44740234$, $\alpha =
(253,2904)$ and $\ell = 2$. The elliptic curve $A$ has CM by the full
ring of integers in $\Q(\sqrt{-7})$, and 2 splits in this ring. A
computation using MAGMA shows that the conditions in the
remark following Theorem \ref{cmmainprop} are
satisfied and thus the conclusion of Theorem~\ref{cmmain} holds. By
Theorem~\ref{cmcomp} and Theorem~\ref{interp} we have that
$\overline{\alpha}$ has odd order for $(1/9 + (1/3))/2 = 2/9 \approx
0.222$ of primes $p$.
\end{ex}

\begin{ex}
\label{cmramified} Let $K = \Q$, $A : y^{2} = x^{3} + 3x$,
$\alpha = (1,-2)$ and $\ell = 2$. The elliptic curve $A$ has CM
by $\Z[i]$ and in this case $\ell$ is ramified. A lengthy
computation shows that the image of $\omega$ has index 4 in
$\Z_{2}^{2} \rtimes H$, where
\[
H = \left\{ \left[\begin{matrix} a & b \\ \mp b & \pm a
\end{matrix}\right] : a, b \in \Z_{2}, a^{2} + b^{2} \equiv 1 \pmod{2}
\right\}
\]
is the corresponding Cartan normalizer. The image of $\omega_{2}$
is generated by
\[
  \left((1,1), \left[ \begin{matrix} 1 & 0\\ 0 & -1
  \end{matrix}\right]\right),
  \left((0,0), \left[ \begin{matrix} 0 & -1\\ 1 & 0 \end{matrix}
  \right]\right), \left((1,1), \left[ \begin{matrix} 2 & -1\\ 1 &
  2 \end{matrix}\right]\right).
\]
One can compute that in this case $\f(G) = \frac{17}{32} \approx
0.531$.
\end{ex}

\section{Higher-Dimensional Abelian Varieties} \label{abvar}

If the abelian algebraic group $A$ is projective, then $A$ is an
abelian variety. In this section we will describe the case when
$\dim(A) > 1$. Assume that $\phi = [\ell]$, the multiplication by
$\ell$ map and let $d = \dim(A)$.

To determine the image of $\omega$ it is crucial to know about the
image of $\rho : \Gal(T_{\infty}/F) \hookrightarrow
\GL_{2d}(\Z_{\ell})$.

The Weil $e_{m}$-pairing is a nondegenerate, skew-symmetric,
Galois invariant pairing $e_{m} : A[m] \times \hat{A}[m] \to
\mu_{m}$. If $\Phi : A \to \hat{A}$ is a polarization defined
over $K$, then the pairing $e_{m,\Phi} : A[m] \times A[m] \to
\mu_{m}$ given by $e_{m,\Phi}(a,b) = e_{m}(a,\Phi(b))$ is
skew-symmetric and Galois invariant. Moreover, it is
nondegenerate provided that $m$ is coprime to $\# \ker(\Phi)$.
The Galois invariance and non-degeneracy implies that $\mathcal{T}_{n} \subseteq
\GSp_{2d}(\Z/\ell^{n} \Z)$, the group of symplectic similitudes. For more
background about abelian varieties, see \cite{jhsdioph}, section A.7.

We have the following surjectivity criteria for $\rho$. 
\begin{prop}
\label{abvarrho}
Let $\ell$ be a prime, $d \geq 2$ and assume that $\gcd(\ell, \#
\ker(\Phi)) = 1$. Then, the $\ell$-adic representation $\rho :
\Gal(T_{\infty}/F) \to \GSp_{2d}(\Z_{\ell})$ is surjective if and only if
the following conditions hold:
\begin{enumerate}
\item $F$ is linearly disjoint from $\Q(\zeta_{\ell^{n}})$
  for all $n$.
\item $\Gal(T_{1}/F) \cong \GSp_{2d}(\Z/\ell \Z)$.
\item If $\ell = d = 2$, then
$T_{1}$ is linearly disjoint from $\Q(\sqrt{2}, i)$.
\end{enumerate}
\end{prop}
\begin{proof}
This is a restatement of Vasiu's Theorems 4.1 and 4.2.1 from \cite{Vasiu}.
\end{proof}

\begin{rem}
  Suppose that $d$ is odd, $d = 2$ or $d = 6$, and $\End(A) \cong \Z$.
  Th\'eor\`eme 3 of \cite[R\'esum\'e des cours de
  1985-1986]{Serre8586} 
  implies that the conditions of the above proposition are satisfied for
  $\ell$ sufficiently large.
\end{rem}

The following result gives criteria for when the map to the Kummer
part is surjective.

\begin{thm}
\label{abvarkappa}
Let $\ell$ be prime, $d \geq 2$ and assume that $\gcd(\ell,\#
\ker(\Phi)) = 1$, and the $\ell$-adic representation $\rho : \Gal(T_{\infty}/F)
\to \GSp_{2d}(\Z_{\ell})$ is surjective.
Then the Kummer map $\kappa :
\Gal(K_{\infty}/T_{\infty}) \to \Z_{\ell}^{2d}$ is surjective
if and only if the following conditions hold:
\begin{enumerate}
\item $\alpha \not\in \ell A(F)$,
\item if $\ell = 2$, $\beta_{1} \not\in A(T_{1})$.
\end{enumerate}
\end{thm}
\begin{proof}
When $\ell > 2$, the only modification necessary in the proof of
Theorem~\ref{surj2} is in showing that
$\Hom_{\mathcal{T}_{1}}(N^{(n)}/N^{(n+1)}, A[\ell]) = 0$. To justify
such a statement, one can use the computation of
Liebeck and Seitz (see Proposition 1.10 of \cite{LiebeckSeitz})
of the composition factors of this module over $\overline{\F}_{\ell}$,
combined with the Restriction Theorem (see the theorem in 
Section 2.11 of Humphreys' book \cite{Humph}) to conclude that
these composition factors are still irreducible over $\F_{\ell}$.
We find that $N^{(n)}/N^{(n+1)}$ is a one-dimensional extension of an
irreducible $\mathcal{T}_{1}$-module of dimension $2g^{2} + g$, so
again $\Hom_{\mathcal{T}_{1}}(N^{(n)}/N^{(n+1)}, A[\ell]) = 0$.

When $\ell = 2$, we assume that $\beta_{1} \not\in A(T_{1})$.  We seek to
apply Lemma~\ref{surjlem2}. In this case, $V = N^{(n)}/N^{(n+1)}$ has
a natural submodule of dimension $V_{1} = 2g^{2} - g$. In order to conclude that
$\Hom_{\mathcal{T}_{1}}(N^{(n)}/N^{(n+1)}, A[2]) = 0$, one shows that
any submodule $M$ of $V$ with $V/M \cong A[\ell]$ must contain $V_{1}$.
This implies that $V_{1}$ has codimension one in $M$, and it can be
checked that no such submodule $M$ exists. Thus, the hypotheses of 
Lemma~\ref{surjlem2} are satisfied, and we can conclude that
$\alpha \not \in A(F) \cap \ell A(T_{n})$ for any $n$. Then
Theorem~\ref{surj1} implies that $\omega_{n}$ is surjective.
\end{proof}

\begin{rem}
When $\ell = 2$, $\GSp_{2d}(\F_{2}) = \Sp_{2d}(\F_{2})$ is simple (provided
$d \geq 3$) and so Lemma~\ref{surjlem3} does not apply. Indeed, suppose
that $\alpha \in 2 A(T_{1})$, but $\alpha \not\in 2 A(F)$.
This means that $\delta_{1}(\alpha)$ lies in the kernel
of the restriction map $H^{1}(F, A[2]) \to H^{1}(T_{1}, A[2])$.
However, the exactness of
\[
\begin{CD}
0 @>>> H^{1}(\Gal(T_{1}/K), A[2]) @>>> H^{1}(F, A[2]) @>>>
H^{1}(T_{1}, A[2])
\end{CD}
\]
implies that the kernel is $H^{1}(\Gal(T_{1}/K), A[2])$, which is
shown to be isomorphic to $\Z/2\Z$ by Pollatsek in
\cite{Pollatsek}. It follows from the explicit construction of the
non-trivial cocycles that $\alpha \in 2 A(T_{1})$ if and only if the
preimages of $\alpha$ are a union of two Galois orbits of size
$2^{2d-1} + 2^{d-1}$ and $2^{2d - 1} - 2^{d - 1}$, respectively,
corresponding to the subgroups $\SO^{+}_{2d}(\F_{2})$ and
$\SO^{-}_{2d}(\F_{2})$ stabilizing the two isomorphism classes of
quadratic forms of dimension $2d$. It is interesting to ask
whether there are abelian varieties $A/\Q$ and $\alpha \in A(\Q) - 2A(\Q)$
for which this occurs.
\end{rem}

\begin{cor}
The arboreal representation $\omega :
  \Gal(K_{\infty}/F) \to (\Z_\ell)^{2d} \rtimes \GSp_{2d}(\Z_\ell)$ is
  surjective if and only if the conditions of Theorem \ref{abvarkappa} and
  Proposition \ref{abvarrho} are satisfied.
\end{cor}

\begin{ex}
\label{abvarex} Let $C$ be the hyperelliptic curve with affine
model $y^{2} = f(x)$, where $f(x) = 4x^{6} - 8x^{5} + 4x^{4} +
4x^{2} - 8x + 5$ and let $A = \Jac(C)$. In \cite[p.~2]{Flynn}
a non-singular model for $C$ is given by
\begin{align*}
  Y^{2} &= 5 X_{0}^{2} - 8 X_{0} X_{1} + 4 X_{1}^{2} + 4 X_{2}^{2} - 8
  X_{2} X_{3} + 4 X_{3}^{2}\\
  X_{0} X_{2} &= X_{1}^{2}, \qquad X_{0} X_{3} = X_{1} X_{2}, \qquad X_{1} X_{3} = X_{2}^{2}.
\end{align*}
The two points at infinity are at $(X_{0} : X_{1} : X_{2} : X_{3} : Y)
= (0 : 0 : 0 : 1 : -2)$ and $(0 : 0 : 0 : 1 : 2)$. Denote the first by
$\infty^{+}$. Let $P = (1 : 1 : 1 : 1 :
1)$ and let $\alpha = \infty^{+} - P \in A(\Q)$.

\begin{prop}
With $A$ and $\alpha$ given above, we have
\[
  \Gal(K_{\infty}/F) \cong (\Z_{\ell})^{4} \rtimes \GSp_{4}(\Z_{\ell})
\]
for all primes $\ell$.
\end{prop}
\begin{proof}
  It suffices to verify the conditions of Theorem~\ref{abvarkappa} and
  Proposition~\ref{abvarrho}. Note that since $J = \Jac(C)$, $J$ is
  endowed with a canonical principal polarization, so $\# \ker(\Phi) =
  1$.

  Next, we check condition (1) of Theorem~\ref{abvarkappa}. The Kummer
  surface $K$ associated to $A$ is $A / \langle [-1] \rangle$.  It is
  a quartic curve in $\P^{3}$ with nodes at the images of $A[2]$, the
  fixed points of $[-1]$. Multiplication by $[m]$ descends to a
  morphism of $K$, and one may use the map $\phi : A \to K$ to define
  a height function $h : A \to \R$ on $A$. Let $\hat{h}$ denote the
  corresponding canonical height. One may use MAGMA to verify that for
  all $P \in A(\Q)$, $|h(P) - \hat{h}(P)| \leq 3.10933$ and that
  $\hat{h}(\alpha) = 0.247060$. Suppose to the contrary that there is
  a prime $\ell$ and $\beta \in A(\Q)$ with $\ell \beta = \alpha$.
  Then, $\hat{h}(\beta) = \frac{1}{\ell^{2}} \hat{h}(\alpha)$ and
  hence $|h(\beta)| \leq 3.10933 + 0.247060$. Computing all points $P
  \in J(\Q)$ satisfying the above bound, we find that there are no
  such $\beta$.

Condition (1) of Proposition~\ref{abvarrho} is obvious.

Next, we check condition (3) of Proposition~\ref{abvarrho}. Since
$\Q(A[2])/\Q$ has Galois group $S_{6}$, there is a unique quadratic
subfield of $\Q(A[2])$, and computing the discriminant of $f(x)$, we
find it to be $\Q(\sqrt{-3 \cdot 13 \cdot 31})$. Hence, $\Q(A[2])$ is
linearly disjoint from $\Q(\sqrt{2}, i)$, as desired.

Next, we check condition (2) of Proposition~\ref{abvarrho}. In
\cite{Dieulefait}, Dieulefait indicates how one can check that the mod
$\ell$ Galois representations associated to an abelian surface $A$
with $\End(A) \cong \Z$ are surjective at all but finitely many
primes, conditional on Serre's conjecture. To show that $\End(A) \cong \Z$,
one can compute that the two-torsion points of $A[2]$ are the Weierstrass
points, and so $\Q(A[2])$ is the splitting field of $f(x)$. This has Galois
group isomorphic to $S_{6} \cong \GSp_{4}(\F_{2})$. Hence,
Proposition~\ref{abvarrho} implies that the 2-adic Galois representation
is surjective. The injectivity of the map
\[
  \End(A) \otimes \Z_{\ell} \to \End_{\Z_{\ell}}(T_{\ell}(A)) \cong \Z_{\ell}
\]
implies that $\End(A)$ has rank 1 and so $\End(A) \cong \Z$. Using the
algorithm of Liu (\cite{Liu}), we find that the conductor of $A$
divides $2^{4} \cdot 3^{5} \cdot 13 \cdot 31$. We use Dieulefait's
recipe and the explicit computation of the characteristic polynomials
of the images of Frobenius in $\Aut(A[\ell])$ afforded by MAGMA. We
find that at all primes $\ell > 7$ of good reduction, the mod-$\ell$
representation is surjective conditional on Serre's
conjecture. Further, explicit computations mod 3, 5, 7, 13 and 31 show
that the mod $\ell$ representation is surjective there as well. We
remark that Serre's conjecture has been proven thanks to work of Khare
and Wintenberger \cite{KW1} and \cite{KW2}, and Kisin
\cite{Kisin}. 

Finally, we check condition (2) of Theorem~\ref{abvarkappa}. In
Appendix I to \cite{Flynn}, Cassels and Flynn make explicit the
morphism on the Kummer surface $K$ induced by the multiplication by 2
map on $A$. One can check that the image $\phi(\alpha)$ of $\alpha$ on
$K$ is $(0 : 1 : 1 : -4)$. Using this, one may compute the preimages
on $K$ of the point $\phi(\alpha)$, which corresponds to $\alpha \in
A(\Q)$. This gives rise to a system of four quartic equations in four
unknowns. Using MAGMA's Gr\"obner basis routine to solve the
corresponding system of algebraic equations, we find that the sixteen
preimages are of the form $(1 : a_{1} : a_{2} : a_{3})$.  Here $a_{1},
a_{2}$ and $a_{3}$ generate $\Q(\beta)$ where $\beta$ has minimal
polynomial
\begin{align*}
  g(x) &= x^{16} - 12x^{14} - 36x^{13} + 316x^{12} - 912x^{11} +
  1412x^{10} - 472x^{9} - 1764x^{8}\\
  &+ 3544x^{7} - 4104x^{6} + 3912x^{5} - 3588x^{4} - 5888x^{3} +
  8232x^{2} - 4576x + 884.
\end{align*}
It follows that the preimages of $(0 : 1 : 1 : -4)$ lie in degree
16 extensions of $\Q$ and hence $[\Q(\beta_{1}) : \Q] = 16$.
Hence, we cannot have $\beta_{1} \in \Q(A[2])$ since
$\Gal(\Q(A[2])/\Q) \cong S_{6}$ has no subgroups of index 16.
Thus condition (6) holds. It follows that the splitting field
of $g(x)$ is $K_{1}$ and so the Galois group of $g(x)$ is
isomorphic to $(\Z/2\Z)^{4} \rtimes \GSp_{4}(\Z/2\Z)$.
\end{proof}
\end{ex}

\begin{rem}
As far as the authors are aware, the curve $C$ given above is the
first example of a hyperelliptic curve of genus 2 for which all of the
$\ell$-adic Galois representations associated to $\Jac(C)$
are surjective.
\end{rem}

Unfortunately, we have been unable to exactly compute the
corresponding densities for the groups $\Z_{\ell}^{4} \rtimes
\GSp_{4}(\Z_{\ell})$. The nature of the explicit method employed in
Theorem~\ref{gl2den} seems unlikely to be fruitful. Here is a table of
bounds computed from conjugacy class information
for $\GSp_{4}(\Z/\ell^{n} \Z)$.

\medskip

\begin{tabular}{cccc}
$\ell$ & Lower bound & Upper bound & $n$ used\\
\hline\\
$2$ & $\frac{26701}{46080} \hspace{0.1 in} (\approx 0.579)$ &
$\frac{1201}{2048} \hspace{0.1 in} (\approx 0.586)$ & $4$\\
\\
$3$ & $\frac{70769}{103680}  \hspace{0.1 in} (\approx 0.683)$ &
$\frac{27203}{38880}  \hspace{0.1 in} (\approx 0.700)$ & $2$\\
\\
\end{tabular}

In general, if $\ell$ is prime and $G_{\phi}(\alpha) = \Z_{\ell}^{4}
\rtimes \GSp_{4}(\Z_{\ell})$, we have
\[
  \frac{\ell^{7} - 2 \ell^{6} - \ell^{5} + 4 \ell^{4} - 2 \ell^{3} + 2
    \ell^{2} - 5}{(\ell^{4} - 1) (\ell^{2} - 1) (\ell - 1)}
\leq \f(G) \leq \frac{\ell^{7} - \ell^{6} - \ell^{5} + 3 \ell^{4} -
2 \ell^{3} + \ell^{2} - 4}{\ell^{7} - \ell^{5} - \ell^{3} + \ell}.
\]
These follow from the computation of the number of $M \in
\GSp_{4}(\F_{\ell})$ with $\det(M-I) \not\equiv 0 \pmod{\ell}$ in
\cite[p.~61]{kuhlman}.

\appendix

\section{A result relating to Question \ref{higherdim}}
\begin{center} by Jeffrey D. Achter\footnote{Partially supported by NSA grant H98230-08-1-0051.} \end{center}

Fix an odd prime $\ell$.  This appendix provides a proof of:

\begin{prop}
\label{propmain}
The limit $\lim_{g \rightarrow \infty} \calf(
\Z_\ell^{2g}\rtimes \gsp_{2g}(\Z_\ell))$ exists.
\end{prop}

The proof requires some notation concerning symplectic groups.  Let
$\ell$ be a fixed prime.  For each natural number $g$, fix a free
$\Z_\ell$-module $V_g$ of rank $2g$, equipped with a symplectic
pairing $\langle\cdot,\cdot\rangle$.  For each natural number $n$, let
$V_{g,n} = V_g \otimes_{\Z_\ell} \Z_\ell/\ell^n\Z_\ell$.  After a
choice of basis of $V_g$, we have $\gsp_{2g}(\Z_\ell/\ell^n\Z_\ell) \cong
\gsp(V_{g,n},\langle\cdot,\cdot\rangle)$.  For natural numbers $n\ge m$, let
$\rho_{g,n,m}:\gsp_{2g}(\Z_\ell/\ell^n\Z_\ell) \rightarrow
\gsp_{2g}(\Z_\ell/\ell^m\Z_\ell)$ and $\rho_{g,n}:\gsp_{2g}(\Z_\ell)
\rightarrow \gsp_{2g}(\Z_\ell/\ell^n\Z_\ell)$ be the usual reduction maps.  For
any ring $\Lambda$ there is a group homomorphism $\mult:
\gsp_{2g}(\Lambda) \rightarrow \Lambda\units$, and $\sp_{2g}(\Lambda) =
\mult\inv(1)$.  If $m \in \Lambda\units$ and $S \subseteq
\gsp_{2g}(\Lambda)$, let $S^{(m)} = S\cap \mult\inv(m)$.

Since a matrix over $\Z_\ell/\ell^n\Z_\ell$ is invertible if 
and only if its reduction modulo $\ell$ is, 
\begin{align*}
\#\gl_g(\Z_\ell/\ell^n\Z_\ell) &=
\ell^{(n-1)g^2} \#\gl_g(\Z_\ell/\ell\Z_\ell)\\
&= \ell^{(n-1)g^2} \prod_{j=1}^g \ell^{j-1}(\ell^j-1).
\intertext{If $n \ge 2$, a direct calculation shows that $\ker \rho_{g,n,n-1}$ is
isomorphic to the Lie algebra $\mathfrak s \mathfrak p_{2g,
\Z_\ell/\ell\Z_\ell}$, so that for $n\ge 1$ we have
}
\#\sp_{2g}(\Z_\ell/\ell^n\Z_\ell) &= \ell^{(n-1)(2g^2+g)}
\#\sp_{2g}(\Z_\ell/\ell\Z_\ell)\\
&= \ell^{(n-1)(2g^2+g)}\prod_{j=1}^g \ell^{2j-1}(\ell^{2j}-1).
\end{align*}
Since $\mult$ is surjective, 
$\#\gsp_{2g}(\Z_\ell/\ell^n \Z_\ell) =
\#((\Z_\ell/\ell^n\Z_\ell)\units)
\#\sp_{2g}(\Z_\ell/\ell^n\Z_\ell)$. 

For $0 \le r \le g$ define 
\begin{align}
\label{eqs}
S(g,r,n) &= \frac{\#\sp_{2g}(\Z_\ell/\ell^n\Z_\ell)}{ \#
  \sp_{2r}(\Z_\ell/\ell^n\Z_\ell) \#
  \sp_{2(g-r)}(\Z_\ell/\ell^n\Z_\ell)} \\ 
\label{eql}
L(g,n) &= \frac{\#\sp_{2g}(\Z_\ell/\ell^n\Z_\ell)}{\#\gl_g(\Z_\ell/\ell^n\Z_\ell)
  \cdot \#\gl_g(\Z_\ell/\ell^n\Z_\ell)},
\end{align}
with the convention that for $g =0$,
$\sp_{2g}(\Z_\ell/\ell^n\Z_\ell)$ and $\gsp_{2g}(\Z_\ell/\ell^n\Z_\ell)$ are the
trivial group.
Then $S(g,r,n)$ is the number of decompositions $V_{g,n} = E\oplus W$ 
where $E \cong V_{r,n}$ and $W\cong V_{g-r,n}$, while $L(g,n)$ is the 
number of decompositions $V_{g,n} = E\oplus W$ where $E$ and $W$ are 
each Lagrangian.  

For $x \in \gsp_{2g}(\Z_\ell/\ell^n\Z_\ell)$, let
\begin{equation*}
\epsilon(x) = \min \st{ \ord_\ell(\det(\til x- \id)) : \til x \in \rho_{g,n}\inv(x)}.
\end{equation*}
Set
\begin{equation*}
F(g,n) = \oneover{\#\gsp_{2g}(\Z_\ell/\ell^n\Z_\ell)} \sum_{x \in
  \gsp_{2g}(\Z_\ell/\ell^n\Z_\ell)} \ell ^{-\epsilon(x)}.
\end{equation*}

\begin{lem}
\label{lemfiniteokay}
For each $g$ and $n$, $\abs{\calf(\Z_\ell^{2g}
  \rtimes \gsp_{2g}(\Z_\ell)) - F(g,n)} <
\ell^{-n}$.
\end{lem}

\begin{proof}
Let $C_{g,n} = \st{ x \in \gsp_{2g}(\Z_\ell/\ell^n\Z_\ell) : \epsilon(x) <
  n}$.  If $x \in C_{g,n}$ and if $\til x \in \rho_{g,n}\inv(x)$, then $\ord_\ell(\det(\til x
- \id)) = \epsilon(x)$.  Let $\til D_{g,n} = \gsp_{2g}(\Z_\ell) -
\rho_{g,n}\inv(C_{g,n})$.
By Theorem \ref{matrixprop}, we have
\begin{align*}
\abs{F(g,n) - \calf(\Z_\ell^{2g} \rtimes
  \gsp_{2g}(\Z_\ell))} &=
\int_{\til D_{g,n}}(\ell^{-n} - \ell^{-\ord_\ell(\til x)})d\mu \\
&\le \ell^{-n} \mu(\til D_{g,n}) < \ell^{-n}.
\end{align*}
\end{proof}

If $x \in \sp_{2g}(\Z_\ell/\ell\Z_\ell)$, then its
characteristic polynomial $f_x(T)$ is self-reciprocal.  More generally, if $x
\in \gsp_{2g}(\Z_\ell/\ell\Z_\ell)$ has multiplier $\mult(x) =
m$, then the roots (over the algebraic closure of
$\Z_\ell/\ell\Z_\ell$) of $f_x(T)$ may be
arranged in $g$ pairs $\st{\alpha, m/\alpha}$.

If $x \in \gsp_{2g}(\Z_\ell/\ell^n\Z_\ell)$, let
$\overline{\mult(x)} \in (\Z_\ell/\ell\Z_\ell)\units$ be the
reduction of its  multiplier modulo $\ell$.  
Define subsets of $\gsp_{2g}(\Z_\ell/\ell^n\Z_\ell)$
\begin{align*}
\calu_{g,n} &= \st{ x\in \gsp_{2g}(\Z_\ell/\ell^n\Z_\ell): \text{each
    eigenvalue of $\rho_{g,n,1}(x)$ is }1\text{ or } \overline{\mult(x)}}\\
&= \st{ x \in \gsp_{2g}(\Z_\ell/\ell^n\Z_\ell) : f_x(T) \equiv
(T-1)^g(T-\mult(x))^g \bmod \ell} \\ 
\caln_{g,n} &= \st{ x \in \gsp_{2g}(\Z_\ell/\ell^n\Z_\ell) : \rho_{g,n,1}(x)-\id\text{ is
    invertible}}\\
&= \st{ x \in \gsp_{2g}(\Z_\ell/\ell^n\Z_\ell) : f_x(1)\not
\equiv 0 \bmod \ell}
\end{align*}
and quantities
\[
a_{g,n}^{(m)} =
\frac{\#\calu_{g,n}^{(m)}}{\#\sp_{2g}(\Z_\ell/\ell^n\Z_\ell)}
\quad
b_{g,n}^{(m)} =
\frac{\#\caln_{g,n}^{(m)}}{\#\sp_{2g}(\Z_\ell/\ell^n\Z_\ell)}
\quad
d_{g,n}^{(m)} =
\oneover{\#\sp_{2g}(\Z_\ell/\ell^n\Z_\ell)}\sum_{x \in \calu_{g,n}^{(m)}}\ell^{-\epsilon(x)}
\]
for each $m \in (\Z_\ell/\ell^n\Z_\ell)\units$.   We adopt the convention that for $g = 0$,
$\calu_{0,n} = \caln_{0,n} = \gsp_{0}(\Z_\ell/\ell^n\Z_\ell)$.  In particular,
$a_{0,n}^{(m)} = b_{0,n}^{(m)} = 1$.

While this notation is convenient, in fact the quantities $a_{g,n}^{(m)}$ and $b_{g,n}^{(m)}$ are independent of $n$, in the following sense.

\begin{lem}
\label{lemindepn}
Suppose $g$ and $n$ are natural numbers with $n \ge 2$, and that $m
\in (\Z/\ell^n)\units$.  Let $\overline m$ be the class of $m$ modulo
$\ell$.  Then $a_{g,n}^{(m)} = a_{g,1}^{(\overline m)}$ and $b_{g,n}^{(m)}
= b_{g,1}^{(\overline m)}$.
\end{lem}

\begin{proof}
It suffices to prove that if $\overline x \in
\gsp_{2g}(\Z_\ell/\ell\Z_\ell)$ with multiplier $\mult(\overline x)
= \overline m$, and if $m$ is any lift of $\overline m$ to
$(\Z_\ell/\ell^n\Z)\units$, then $\#\rho_{g,n,1}\inv(\overline
x)^{(m)}/\#\sp_{2g}(\Z_\ell/\ell^n\Z_\ell) =
1/\#\sp_{2g}(\Z_\ell/\ell\Z_\ell)$.  Since $\rho_{g,n,1}$ is
surjective, $\#\rho_{g,,n,1}\inv(\overline x) = \ell^{(2g^2+g+1)(n-1)}$.
Suppose $m'$ is a second lift of $\overline m$.  Choose $y \in
\rho_{g,n,1}\inv(\id)$ with $\mult(y) = m'm\inv$; then multiplication
by $y$ shows that $\#\rho_{g,n,1}\inv(\overline x)^{(m)} =
\#\rho_{g,n,1}\inv(\overline x)^{(m')}$.  There are $\ell^{n-1}$ different
lifts $m$, and thus $\#\rho_{g,n,1}\inv(\overline
x)^{(m)}/\#\sp_{2g}(\Z_\ell/\ell^n\Z_\ell) = \ell^{-(n-1)}
\ell^{(2g^2+g+1)(n-1)}/\ell^{(2g^2+g)(n-1)}\#\sp_{2g}(\Z_\ell/\ell\Z_\ell)$,
as desired.
\end{proof}

Define generating functions
\begin{align*}
A_n^{(m)}(T) &= \sum_{g \ge 0} a_{g,n}^{(m)}T^g \\
B_n^{(m)}(T) &= \sum_{g \ge 0} b_{g,n}^{(m)}T^g.
\end{align*}

Suppose $x \in \gsp_{2g}(\Z_\ell/\ell^n\Z_\ell) \cong \gsp(V_{g,n})$.  Then $x$ uniquely
determines an $x$-stable decomposition
\begin{equation}
\label{eqdecompv}
V_{g,n} = E_x \oplus W_x,
\end{equation}
where $E_x \cong V_{r,n}$ for some $r$, $W_x \cong V_{g-r,n}$, $x\rest
{E_x} \in \calu_{r,n}$, and $x\rest {W_x} \in
\caln_{g-r,n}$.  We may thus index elements of $\gsp(V_{g,n})$ by
decompositions \eqref{eqdecompv} and suitable choices for
$x\rest{E_x}$ and $x\rest{W_x}$, so that
\begin{equation}
\label{eqdecompgsp}
\#\gsp_{2g}(\Z_\ell/\ell^n\Z_\ell)^{(m)} = \sum_{r=0}^g S(g,r,n) \# \calu_{r,n}^{(m)}
\# \caln_{g-r,n}^{(m)}.
\end{equation}

\begin{lem}
\label{lemacvg}
For each $n \in {\mathbb N}$ and $m \in (\Z_\ell/\ell^n\Z_\ell)\units$, $A_n^{(m)}(T)$ is a
convergent nonvanishing function on a (complex) disk of radius $R>1$.
\end{lem}

\begin{proof}
By Lemma \ref{lemindepn}, it suffices to prove the result for $n=1$.
Recall that if $H$ is a finite group of
Lie type over $\Z_\ell/\ell\Z_\ell$, then the number of unipotent elements in
$H$ is $\ell^{\dim H - \text{rank} H}$ \cite{springersteinberg}.  Therefore,
the number of unipotent
elements in $\gsp_{2g}(\Z_\ell/\ell\Z_\ell)$
is $\ell^{2g^2}$, and the number
of unipotent elements in $\gl_g(\Z_\ell/\ell\Z_\ell)$ is
$\ell^{g^2-g}$.

In particular, $a_{g,1}^{(1)} =
\ell^{g^2}/\prod_{j=1}^g(\ell^{2j}-1)$; an appeal to the ratio test
shows that $A_1^{(1)}(T)$ converges on any disk of radius smaller than
$\ell$.  Moreover, since $a_{0,1}^{(1)} = 1$ and $\ell \ge 3$,
$a_{0,1}^{(1)} > \sum_{g \ge 1} a_{g,1}^{(1)}$ and thus $A_1^{(1)}(T)$ is
nonvanishing on some disk of radius greater than one.

Now suppose $m\in (\Z_\ell/\ell\Z_\ell)\units$ is not one.  If $x \in
\calu^{(m)}_{g,1}$, then there is a decomposition $V_{g,1} = E\oplus
W$ where $E$ and $W$ are Lagrangian subspaces stable under $x$,
$x\rest E$ is unipotent, and $x \rest W$ is uniquely determined by
$\mult(x)$ and $x\rest E$.  The number of decompositions of $V_{g,1}$
as a sum of Lagrangian subspaces is $L(g,1)$, and the number of
choices for $x \rest E$ is $\ell^{g^2-g}$.  Therefore, $a_{g,1}^{(m)}
= 1/\prod_{j=1}^g (\ell^j-1)^2$, and the argument proceeds as before.
\end{proof}

\begin{lem}
\label{lembexists}
Suppose $n \in {\mathbb N}$ and $m \in (\Z_\ell/\ell^n\Z_\ell)\units$.  Then $\lim_{g \rightarrow \infty}
b_{g,n}^{(m)}$ exists.
\end{lem}

\begin{proof}
Using \eqref{eqs}, the decomposition \eqref{eqdecompgsp} shows that
for each $g$, $\sum_{r=0}^g a_{r,n}^{(m)} b_{g-r,n}^{(m)} = 1$.  Therefore, there is an
equality of generating functions
\[
A_n^{(m)}(T)\cdot B_n^{(m)}(T) = \sum_{g\ge 0} T^g = \oneover{1-T}.
\]
By Lemma \ref{lemacvg}, there exists a number $R>1$ such that the
function $C_n^{(m)}(T) := 1/A_n^{(m)}(T)$ is analytic inside $\abs T < R$.  Let $C_n^{(m)}(T) = \sum
c_{g,n}^{(m)} T^g$ be the series expansion of $C$ centered at the origin.  Since
$B_n^{(m)}(T) = C_n^{(m)}(T)/(1-T)$, we have
\[
b_{g,n}^{(m)} = \sum_{j=1}^g c_{g,j}^{(m)}.
\]
Since $C_n^{(m)}(1)$ is well-defined, $\lim_{g \rightarrow \infty}b_{g,n}^{(m)} = C_n^{(m)}(1)$ exists.
\end{proof}

\begin{proof}[Proof of Proposition \ref{propmain}] 
By Lemma \ref{lemfiniteokay}, it suffices to show that for each $n$,
$\lim_{g \rightarrow \infty} F(g,n)$ exists.  Suppose $x \in \gsp(V_{g,n})$;
write $V_{g,n} = E_x\oplus W_x$ as in \eqref{eqdecompv}.  Then
$\epsilon(x) = \epsilon(x\rest{E_x})$.  Therefore, we may compute
$F(g,n)$ as
\begin{align*}
F(g,n) &= 
\oneover{\#\gsp_{2g}(\Z_\ell/\ell^n\Z_\ell)} \sum_{x \in \gsp_{2g}(\Z_\ell/\ell^n\Z_\ell)}
\ell^{-\epsilon(x)}\\
&= \oneover{\#\gsp_{2g}(\Z_\ell/\ell^n\Z_\ell)} \sum_{m \in (\Z_\ell/\ell^n\Z_\ell)\units} \sum_{r=0}^g
S(g,r,n) \# \caln_{g-r,n}^{(m)} \sum_{x \in \calu_{r,n}^{(m)}}
\ell^{-\epsilon(x)} \\
&= \oneover{\#(\Z_\ell/\ell^n\Z_\ell)\units} \sum_{m \in (\Z_\ell/\ell^n\Z_\ell)\units} \sum_{r=0}^g
\frac{\#\caln_{g-r,n}^{(m)}}{\#\sp_{2(g-r)}(\Z_\ell/\ell^n\Z_\ell)} \cdot
\oneover{\#\sp_{2r}(\Z_\ell/\ell^n\Z_\ell)} \sum_{x \in
\calu_{r,n}^{(m)}}\ell^{-\epsilon(x)}.
\end{align*}
Since for fixed $n$ there are finitely many choices for $m$, it
suffices to show that $\lim_{g \rightarrow \infty} \sum_{r=0}^g
b_{g-r,n}^{(m)} d_{r,n}^{(m)}$ exists.  This follows from the existence
(Lemma \ref{lembexists}) of $\lim_{g \rightarrow \infty} b_{g,n}^{(m)}$, and
the fact that each term of $\sum_{g \ge 0} d_{g,n}^{(m)}$ is smaller
than the corresponding term in the convergent (Lemma \ref{lemacvg})
series $\sum_{g \ge 0} a_{g,n}^{(m)}$.
\end{proof}

\section{Numerical Data}

In this appendix, we give numerical data related to the examples
given in the paper. Each table below includes several choices of
$x$, the number of primes $\leq x$ where $\alpha$ (and/or $A$)
has good reduction (total primes), and the number of such primes
where the order of $\alpha$ is coprime to $\ell$ (good primes),
and the ratio.

The following is data for Example~\ref{untwistedtorus}, $A :
x^{2} - y^{2} = 1$, with $\ell = 2$ and $\alpha =
\left(\frac{5}{3}, \frac{4}{3}\right)$.

\begin{tabular}{c|cccccc}
$x$ & $10^{3}$ & $10^{4}$ & $10^{5}$ & $10^{6}$ & $10^{7}$ & $\infty$\\
\hline
Good primes & $57$ & $406$ & $3197$ & $26200$ & $221805$ & \\
Total primes & $167$ & $1228$ & $9591$ & $78497$ & $664578$ & \\
Ratio & $.34132$ & $.33062$ & $.33333$ & $.33377$ & $.33375$ & $.33333$\\
\end{tabular}

The following is data for Example~\ref{badtwist}, $A : x^{2} +
7y^{2} = 1$, $\ell = 7$ and $\alpha = \left(\frac{3}{4},
\frac{1}{4}\right)$.

\begin{tabular}{c|cccccc}
$x$ & $10^{3}$ & $10^{4}$ & $10^{5}$ & $10^{6}$ & $10^{7}$ & $\infty$\\
\hline
Good primes & $115$ & $870$ & $6805$ & $55608$ & $470765$ & \\
Total primes & $167$ & $1228$ & $9591$ & $78497$ & $664578$ & \\
Ratio & $.68862$ & $.70847$ & $.70952$ & $.70841$ & $.70837$ & $.70833$\\
\end{tabular}

The following is data for Example~\ref{bigtorus}, $A : x^{3} +
2y^{3} + 4z^{3} - 6xyz = 1$, with $\ell = 2$ and $\alpha =
(-1,1,0)$.

\begin{tabular}{c|cccccc}
$x$ & $10^{3}$ & $10^{4}$ & $10^{5}$ & $10^{6}$ & $10^{7}$ & $\infty$\\
\hline
Good primes & $62$ & $492$ & $3840$ & $31353$ & $265226$ & \\
Total primes & $168$ & $1229$ & $9592$ & $78498$ & $664579$ & \\
Ratio & $.36905$ & $.40033$ & $.40033$ & $.39941$ & $.39909$ & $.39881$\\
\end{tabular}

The following is data for Example~\ref{noncmex}, $A : y^{2} + y =
x^{3} - x$, with $\ell = 2$ and $\alpha = (0,0)$.

\begin{tabular}{c|cccccc}
$x$ & $10^{3}$ & $10^{4}$ & $10^{5}$ & $10^{6}$ & $10^{7}$ & $\infty$\\
\hline
Good primes & $93$ & $654$ & $5029$ & $41080$ & $348035$ & \\
Total primes & $167$ & $1228$ & $9591$ & $78497$ & $664578$ & \\
Ratio & $.55689$ & $.53257$ & $.52434$ & $.52333$ & $.52369$ & $.52381$\\
\end{tabular}

The following is data for Example~\ref{cmnonsplit}, $A : y^{2} =
x^{3} + 3$, $\ell = 2$ and $\alpha = (1,2)$.

\begin{tabular}{c|cccccc}
$x$ & $10^{3}$ & $10^{4}$ & $10^{5}$ & $10^{6}$ & $10^{7}$ & $\infty$\\
\hline
Good primes & $90$ & $670$ & $5093$ & $41868$ & $354068$ & \\
Total primes & $166$ & $1227$ & $9590$ & $78496$ & $664577$ & \\
Ratio & $.54217$ & $.54605$ & $.53107$ & $.53338$ & $.53277$ & $.53333$\\
\end{tabular}

The following is data for Example~\ref{cmsplit}, $A : y^{2} = x^{3} -
207515x + 44740234$, $\ell = 2$ and $\alpha = (253,2904)$.

\begin{tabular}{c|cccccc}
$x$ & $10^{3}$ & $10^{4}$ & $10^{5}$ & $10^{6}$ & $10^{7}$ & $\infty$\\
\hline
Good primes & $39$ & $269$ & $2113$ & $17407$ & $147714$ & \\
Total primes & $165$ & $1226$ & $9589$ & $78495$ & $664576$ & \\
Ratio & $.23636$ & $.21941$ & $.22036$ & $.22176$ & $.22227$ &
$.22222$\\
\end{tabular}

The following is data for Example~\ref{cmramified}, $A : y^{2} =
x^{3} + 3x$, $\ell = 2$ and $\alpha = (1,-2)$.

\begin{tabular}{c|cccccc}
$x$ & $10^{3}$ & $10^{4}$ & $10^{5}$ & $10^{6}$ & $10^{7}$ & $\infty$\\
\hline
Good primes & $89$ & $663$ & $5082$ & $41757$ & $353023$ & \\
Total primes & $166$ & $1227$ & $9590$ & $78496$ & $664577$ & \\
Ratio & $.53614$ & $.54034$ & $.52993$ & $.53196$ & $.53120$ & $.53125$\\
\end{tabular}

The following is data for Example~\ref{abvarex}, $A = \Jac(C)$ where
$C : y^{2} = 4x^{6} - 8x^{5} + 4x^{4} + 4x^{2} - 8x + 5$, $\ell = 2$
and $\alpha = \infty^{+} - P$.

\begin{tabular}{c|cccccc}
$x$ & $10^{3}$ & $10^{4}$ & $10^{5}$ & $10^{6}$ & $10^{7}$ & $\infty$\\
\hline
Good primes & $101$ & $725$ & $5584$ & $45832$ & $388144$ & \\
Total primes & $164$ & $1225$ & $9588$ & $78494$ & $664575$ & \\
Ratio & $.61585$ & $.59183$ & $.58239$ & $.58389$ & $.58405$ &
$0.57944 \leq \mathcal{F} \leq 0.58643$\\
\end{tabular}

\section*{Acknowledgements}
  The authors are grateful to Antonella Perucca for her contributions
  to and close reading of the proof of Theorem \ref{interp}, and for
  many useful comments.  We extend a special thanks to the referee for
  a very close and helpful reading of the manuscript, and for an
  extensive set of detailed comments on ways to streamline the
  arguments in the paper.  We would also like to thank Ken Ribet,
  Daniel Bertrand, Wojciech Gajda, Ken Ono, Ram Murty, Nigel Boston,
  and Jordan Ellenberg for helpful discussions and feedback. Finally,
  we have extensively used the computer package MAGMA \cite{Magma} for
  computations.

\end{document}